\documentclass[12pt, a4paper, abstracton, bibliography=totoc]{scrartcl}
\pdfoutput=1

\usepackage{amsmath, amsthm, amsfonts, amssymb}
\usepackage[utf8]{inputenc}
\usepackage[T1]{fontenc}
\usepackage{color}
\usepackage{slashed} % Für die durchgestrichenen Dirac-Operator-Zeichen.
\usepackage[all, cmtip]{xy}
\usepackage{hyperref} % Anklickbare Links und Inhaltsverzeichnis. Und bookmarks im pdf.
\let\counterwithout\relax

\usepackage{chngcntr} % Um die Fußnotennummerierung fortlaufend im ganzen Dokument machen zu koennen.

\usepackage{graphicx} % Damit du größere \cdot definieren kannst. Und zum Bilder einbinden.
\usepackage[english]{babel}
\usepackage{microtype}
\usepackage{bm} % Für Fettschrift im Mathemodus.
\usepackage{mathtools} % Damit du den \mathclap-Befehl nutzen kannst.

\newcommand{\IN}{\mathbb{N}}
\newcommand{\IZ}{\mathbb{Z}}

\newcommand{\IR}{\mathbb{R}}
\newcommand{\IC}{\mathbb{C}}
\newcommand{\IH}{\mathcal{H}}

\newcommand{\injrad}{\operatorname{inj-rad}}
\newcommand{\Rm}{\operatorname{Rm}}

\newcommand{\LLip}{L\text{-}\operatorname{Lip}}
\newcommand{\IB}{\mathfrak{B}}
\newcommand{\IK}{\mathfrak{K}}
\newcommand{\frakD}{\mathfrak{D}}
\newcommand{\frakC}{\mathfrak{C}}

\newcommand{\supp}{\operatorname{supp}}
\newcommand{\diam}{\operatorname{diam}}
\newcommand{\Hom}{\operatorname{Hom}}
\newcommand{\Mat}{\operatorname{Mat}}

\newcommand{\id}{\operatorname{id}}

\newcommand{\image}{\operatorname{im}}

\newcommand{\card}{\#}
\newcommand{\Vect}{\operatorname{Vect}}
\newcommand{\Idem}{\operatorname{Idem}}

\newcommand{\closure}{\operatorname{cl}}

\newcommand{\op}{\mathrm{op}}
\newcommand{\pt}{\mathrm{pt}}

\newcommand{\Frechet}{Fr\'{e}chet }
\newcommand{\Poincare}{Poincar\'{e} }
\newcommand{\Folner}{F{\o}lner }
\newcommand{\Spakula}{\v{S}pakula }
\newcommand{\spinc}{spin$^c$ }

\newcommand*{\largecdot}{\raisebox{-0.25ex}{\scalebox{1.2}{$\cdot$}}}

\DeclareMathOperator{\hatotimes}{\hat{\otimes}}

\theoremstyle{plain}
\newtheorem{thm}{Theorem}[section]
\newtheorem*{thm*}{Theorem}
\newtheorem*{mainthm*}{Main Result}
\newtheorem{cor}[thm]{Corollary}
\newtheorem*{cor*}{Corollary}
\newtheorem{lem}[thm]{Lemma}
\newtheorem{prop}[thm]{Proposition}
\newtheorem{conj}[thm]{Conjecture}
\newtheorem{question}[thm]{Question}

\newtheorem{thmintro}{Theorem}

\theoremstyle{definition}
\newtheorem{defn-alt}[thm]{Definition}
\newtheorem{example-alt}[thm]{Example}
\newtheorem{examples-alt}[thm]{Examples}
\newtheorem{rem-alt}[thm]{Remark}
\newtheorem{nota-alt}[thm]{Notation}

\newenvironment{defn}    
{
	\pushQED{\qed}\begin{defn-alt}}
	{\popQED\end{defn-alt}}

\newenvironment{example}    
{
	\pushQED{\qed}\begin{example-alt}}
	{\popQED\end{example-alt}}
	
\newenvironment{examples}    
{
	\pushQED{\qed}\begin{examples-alt}}
	{\popQED\end{examples-alt}}

\newenvironment{rem}    
{
	\pushQED{\qed}\begin{rem-alt}}
	{\popQED\end{rem-alt}}

\numberwithin{equation}{section}

\counterwithout{footnote}{section}

\makeatletter% Damit man Footnotes ohne Nummer machen kann.
\def\blfootnote{\gdef\@thefnmark{}\@footnotetext}
\makeatother

\begin{document}

\title{Uniform $K$-theory, and \Poincare duality for uniform $K$-homology}
\author{Alexander Engel}
\date{}
\maketitle

\vspace*{-3.5\baselineskip}
\begin{center}
\footnotesize{
\textit{
Fakult\"{a}t f\"{u}r Mathematik\\
Universit\"{a}t Regensburg\\
93040 Regensburg, GERMANY\\
\href{mailto:alexander.engel@mathematik.uni-regensburg.de}{alexander.engel@mathematik.uni-regensburg.de}
}}
\end{center}

\vspace*{-0.5\baselineskip}
\begin{abstract}
%\blfootnote{\textit{Date:} \today.}
%\blfootnote{\textit{$2010$ Mathematics Subject Classification.} Primary:\ ???; Secondary:\ ???, ???.}
%\blfootnote{\textit{Keywords and phrases.} ???}
We revisit \v{S}pakula's uniform $K$-homology, construct the external product for it and use this to deduce homotopy invariance of uniform $K$-homology.

We define uniform $K$-theory and on manifolds of bounded geometry we give an interpretation of it via vector bundles of bounded geometry. We further construct a cap product with uniform $K$-homology and prove \Poincare duality between uniform $K$-theory and uniform $K$-homology on \spinc manifolds of bounded geometry.
\end{abstract}

\tableofcontents

\section{Introduction}

$K$-homology is a generalized homology theory (in the sense of Eilenberg--Steenrod) which is an indispensable tool in modern index theory. It is made such that elliptic operators naturally define classes in it, there is a proof of the Atiyah--Singer index theorem utilizing crucially $K$-homology (see the exposition in Higson--Roe \cite[Chaper~11]{higson_roe}), and the $K$-homology of the classifying space $BG$ of a group $G$ is the domain of the analytic assembly map featuring in the strong Novikov conjecture.

Working on non-compact spaces one can use a locally finite version of $K$-homology in order to have a receptacle for the classes of elliptic operators over such spaces. This version of $K$-homology is employed in the coarse Baum--Connes conjecture. Locally finite $K$-homology of non-compact spaces is applied to the study of compact spaces by considering the universal covers of the compact spaces. But this method discards some information: if we lift a cycle from the compact space to its universal cover, then the lifted cycle will not only be locally finite, but even uniformly locally finite. Hence one might try to refine the method by inventing a uniform version of locally finite $K$-homology.

A uniform version of locally finite $K$-homology was proposed by \Spakula \cite{spakula_thesis,spakula_uniform_k_homology}. He showed that if one has a closed spin manifold, then the Dirac operator of the universal cover of the manifold (equipped with the lifted Riemannian metric and spin structure) will naturally define a class in uniform $K$-homology. This was generalized by the author to the fact that symmetric, elliptic uniform pseudodifferential operators over manifolds of bounded geometry define classes in uniform $K$-homology \cite[Theorem~3.39]{engel_indices_UPDO}. \Spakula further also set up a uniform version of the coarse Baum--Connes conjecture.

The first goal of the present paper is to revisit \v{S}pakula's uniform $K$-homology and to prove additional properties of it. Our main technical result is the construction of an external product for uniform $K$-homology.

\begin{thmintro}[Theorem~\ref{thm:external_prod_homology}]
Let $X,Y$ be locally compact and separable metric spaces of jointly bounded geometry. Then there exists an associative product
\[\times\colon K^u_p(X) \otimes K^u_q(Y) \to K^u_{p+q}(X \times Y)\]
having the same properties as the usual external product in $K$-homology of compact spaces.
\end{thmintro}

This external product is used to conclude that weakly homotopic uniform Fredholm modules define the same uniform $K$-homology class (Theorem~\ref{thm:weak_homotopy_equivalence_K_hom}). This result has the following consequences:
\begin{itemize}
\item The uniform $K$-homology class of a symmetric, elliptic uniform pseudodifferential operator depends only on the principal symbol (\cite[Proposition~3.40]{engel_indices_UPDO}).
\item Uniform $K$-homology is homotopy invariant for uniformly cobounded and proper Lipschitz maps (Theorem~\ref{thm:homotopy_equivalence_k_hom}). This homotopy invariance is then used to relate the rough Baum--Connes conjecture to the usual Baum--Connes conjecture (see Theorem~\ref{thm:BC_equiv_uniform_coarse}), and it is an important ingredient in the proof of \Poincare duality between uniform $K$-theory and uniform $K$-homology.
\end{itemize}

An important ingredient in the index theory on closed manifolds is that the $K$-homology class of any elliptic operator may be represented by the class of a twisted Dirac operator. In the case of \spinc manifolds this can be proved by establishing \Poincare duality between $K$-theory and $K$-homology since the cap product is given by twisting the Dirac operator of the manifold by the vector bundle representing the $K$-theory class.\footnote{Though note that this approach does not give a concrete formula for how to find this vector bundle. It is an important observation of Atiyah and Singer (and later elaborated upon by Baum and Douglas in their geometric picture for $K$-homology) that if the $K$-homology class is given by an elliptic pseudodifferential operator, then one can use the symbol of the operator to get a representative as the class of a twisted Dirac operator.}

In the second part of this paper we will establishing the analogous statements for uniform $K$-homology. We will first introduce uniform $K$-theory by simply defining
\begin{itemize}
\item $K^\ast_u(X) := K_{-\ast}(C_u(X))$, where $C_u(X)$ is the $C^\ast$-algebra of all bounded, uniformly continuous, complex-valued functions on $X$ and $K_\ast(-)$ is operator $K$-theory.
\end{itemize}
The bulk of Section~\ref{sec:uniform_k_th} is devoted to proving an interpretation of uniform $K$-theory via vector bundles of bounded geometry:

\begin{thmintro}[Theorem~\ref{thm:interpretation_K0u}]
Let $M$ be a Riemannian manifold of bounded geometry and without boundary.

Then every element of $K^0_u(M)$ is of the form $[E] - [F]$, where both $[E]$ and $[F]$ are $C_b^\infty$-isomorphism classes of complex vector bundles of bounded geometry over $M$.

Moreover, every complex vector bundle of bounded geometry over $M$ defines naturally a class in $K^0_u(M)$.
\end{thmintro}

Finally, in Section~\ref{sec:poincare_duality} we prove the \Poincare duality result:

\begin{thmintro}[Theorem~\ref{thm:Poincare_duality_K}]
Let $M$ be an $m$-dimensional spin$^c$ manifold of bounded geometry and without boundary.

Then the cap product $- \cap [M] \colon K_u^\ast(M) \to K^u_{m-\ast}(M)$ with its uniform $K$-fundamental class $[M] \in K_m^u(M)$ is an isomorphism.
\end{thmintro}

The results in this paper are an important ingredient for developing the index theory of symmetric, elliptic uniform pseudodifferential operators over manifolds of bounded geometry. This is carried out in \cite[Section~5]{engel_indices_UPDO}.

\paragraph{Acknowledgements} This article contains Sections~3 \& 4 of the preprint \cite{engel_indices_UPDO} which was split up for easier publication. It arose out of the Ph.D.\ thesis \cite{engel_phd} of the author written at the University of Augsburg. I thank the referee for his or her comments.

\section{Uniform \texorpdfstring{$K$}{K}-homology}

In this section we will investigate uniform $K$-homology---a version of $K$-homology that incorporates into its definition uniformity estimates that one usually has for, e.g., Dirac operators over manifolds of bounded geometry (see Example~\ref{ex_Dirac_uniform}). Uniform $K$-homology was introduced by \Spakula \cite{spakula_thesis,spakula_uniform_k_homology}.\footnote{But we have changed the definition slightly, see Section~\ref{sec_changes_defn} for how and why.} We will revisit it and we will prove additional properties (existence of the Kasparov product and homotopy invariance) that are crucially needed later. Furthermore, we will use in Section \ref{sec:rough_BC} the homotopy invariace to deduce useful facts about the rough Baum--Connes assembly map.

\subsection{Definition and basic properties of uniform \texorpdfstring{$K$}{K}-homology}

Let us recall the notion of multigraded Hilbert spaces. This material is basically taken from Higson--Roe \cite[Appendix~A]{higson_roe}.

\begin{itemize}
\item A \emph{graded Hilbert space} is a Hilbert space $H$ with a decomposition $H = H^+ \oplus H^-$ into closed, orthogonal subspaces. This is the same as prescribing a \emph{grading operator} $\epsilon$ whose $\pm 1$-eigenspaces are $H^\pm$ and such that $\epsilon$ is selfadjoint and unitary.
\item If $H$ is a graded space, then its \emph{opposite} is the graded space $H^\op$ whose underlying vector space is $H$, but with reversed grading, i.e., $(H^\op)^+ = H^-$ and $(H^\op)^- = H^+$. This is equivalent to setting $\epsilon_{H^\op} := -\epsilon_H$.
\item An operator on a graded space $H$ is called \emph{even} if it maps $H^\pm$ to $H^\pm$, and it is called \emph{odd} if it maps $H^\pm$ to $H^\mp$. Equivalently, an operator is even if it commutes with the grading operator $\epsilon$ of $H$, and it is odd if it anti-commutes with it.
\end{itemize}

\begin{defn}[Multigraded Hilbert spaces and multigraded operators]
Let $p \in \IN_0$.

A \emph{$p$-multigraded Hilbert space} is a graded Hilbert space which is equipped with $p$ odd unitary operators $\epsilon_1, \ldots, \epsilon_p$ such that $\epsilon_i \epsilon_j + \epsilon_j \epsilon_i = 0$ for $i \not= j$, and $\epsilon_j^2 = -1$ for all $j$.\footnote{Note that a $0$-multigraded Hilbert space is just a graded Hilbert space. We make the convention that a $(-1)$-multigraded Hilbert space is an ungraded one.}

If $H$ is a $p$-multigraded Hilbert space, then an operator on $H$ is called \emph{multigraded} if it commutes with the multigrading operators $\epsilon_1, \ldots, \epsilon_p$ of $H$.
\end{defn}

Let us now recall the usual definition of multigraded Fredholm modules, where $X$ is a locally compact, separable metric space:

\begin{defn}[Multigraded Fredholm modules]
Let $p \in \IZ_{\ge -1}$.

A triple $(H,\rho,T)$ consisting of
\begin{itemize}
\item a separable $p$-multigraded Hilbert space $H$,
\item a representation $\rho\colon C_0(X) \to \IB(H)$ by even, multigraded operators, and
\item an odd multigraded operator $T \in \IB(H)$ such that
\begin{itemize}
\item the operators $T^2 - 1$ and $T - T^\ast$ are locally compact and
\item the operator $T$ itself is pseudolocal
\end{itemize}
\end{itemize}
is called a \emph{$p$-multigraded Fredholm module over $X$}.

Here an operator $S$ is called \emph{locally compact}, if for all $f \in C_0(X)$ the operators $\rho(f) S$ and $S \rho(f)$ are compact, and $S$ is called \emph{pseudolocal}, if for all $f \in C_0(X)$ the operator $[S, \rho(f)]$ is compact.
\end{defn}

To define uniform Fredholm modules we will use the following notion:
\begin{defn}[Uniformly approximable collections of operators]\label{defn:uniformly_approximable_collection}
A collection of operators $\mathcal{A} \subset \IK(L^2(E))$ is said to be \emph{uniformly approximable}, if for every $\varepsilon > 0$ there is an $N > 0$ such that for every $T \in \mathcal{A}$ there is a rank-$N$ operator $k$ with $\|T - k\| < \varepsilon$.
\end{defn}

Let us define
\begin{equation*}
\LLip_R(X) := \{ f \in C_c(X) \ | \ f \text{ is }L\text{-Lipschitz}, \diam(\supp f) \le R \text{ and } \|f\|_\infty \le 1\}.
\end{equation*}

\begin{defn}[{\cite[Definition 2.3]{spakula_uniform_k_homology}}]\label{defn:uniform_operators}
Let $T \in \IB(H)$ be an operator on a Hilbert space $H$ and $\rho\colon C_0(X) \to \IB(H)$ a representation.

We say that $T$ is \emph{uniformly locally compact}, if for every $R, L > 0$ the collection
\[\{\rho(f)T, T\rho(f) \ | \ f \in \LLip_R(X)\}\]
is uniformly approximable.

We say that $T$ is \emph{uniformly pseudolocal}, if for every $R, L > 0$ the collection
\[\{[T, \rho(f)] \ | \ f \in \LLip_R(X)\}\]
is uniformly approximable.
\end{defn}

Note that by an approximation argument we get that the above defined collections are still uniformly approximable if we enlargen the definition of $\LLip_R(X)$ from $f \in C_c(X)$ to $f \in C_0(X)$.

The following lemma states that on proper spaces we may drop the $L$-dependence for uniformly locally compact operators.

\begin{lem}[{\cite[Remark 2.5]{spakula_uniform_k_homology}}]\label{lem:l_uniformly_loc_compact_without_l}
Let $X$ be a proper space. If $T$ is uniformly locally compact, then for every $R > 0$ the collection
\[\{\rho(f) T, T \rho(f) \ | \ f \in C_c(X), \diam(\supp f) \le R \text{ and } \|f\|_\infty \le 1\}\]
is also uniformly approximable (i.e., we can drop the $L$-dependence).
\end{lem}

Note that an analogous lemma for uniformly pseudolocal operators does not hold. We may see this via the following example: if we have an operator $D$ of Dirac type on a manifold $M$ and if $g$ is a smooth function on $M$, then we have the equation $([D,g]u)(x) = \sigma_D(x, dg) u(x)$, where $u$ is a section into the Dirac bundle $S$ on which $D$ acts, $\sigma_D(x, \xi)$ is the symbol of $D$ regarded as an endomorphism of $S_x$ and $\xi \in T^\ast_x M$. So we see that the norm of $[D,g]$ does depend on the first derivative of the function $g$.

\begin{defn}[Uniform Fredholm modules, cf. {\cite[Definition 2.6]{spakula_uniform_k_homology}}]\label{defn:uniform_fredholm_modules}
A Fredholm module $(H, \rho, T)$ is called \emph{uniform}, if $T$ is uniformly pseudolocal and the operators $T^2-1$ and $T - T^\ast$ are uniformly locally compact.
\end{defn}

\begin{example}[{\cite[Theorem 3.1]{spakula_uniform_k_homology}}]
\label{ex_Dirac_uniform}
\Spakula showed that the usual Fredholm module arising from a generalized Dirac operator is uniform if we assume bounded geometry\footnote{A manifold is said to have bounded geometry if its curvature tensor and all its derivatives are uniformly bounded and if its injectivity radius is uniformly positive. A vector bundle equipped with a metric and connection is said to have bounded geometry if its curvature tensor and all its derivatives are uniformly bounded.}: if $D$ is a generalized Dirac operator acting on a Dirac bundle $S$ of bounded geometry over a manifold $M$ of bounded geometry, then the triple $(L^2(S), \rho, \chi(D))$, where $\rho$ is the representation of $C_0(M)$ on $L^2(S)$ by multiplication operators and $\chi$ is a normalizing function, is a uniform Fredholm module.

In \cite[Theorem~3.39]{engel_indices_UPDO} this statement was generalized to symmetric and elliptic uniform pseudodifferential operators over manifolds of bounded geometry.
\end{example}

For a totally bounded metric space uniform Fredholm modules are the same as usual Fredholm modules. Since \v{S}pakula does not give a proof of it, we will do it now:

\begin{prop}\label{prop:compact_space_every_module_uniform}
Let $X$ be a totally bounded metric space. Then every Fredholm module over $X$ is uniform.
\end{prop}

\begin{proof}
Let $(H, \rho, T)$ be a Fredholm module.

First we will show that $T$ is uniformly pseudolocal. We will use the fact that the set $\LLip_R(X) \subset C(X)$ is relatively compact (i.e., its closure is compact) by the Theorem of Arzel\`{a}--Ascoli.\footnote{Since Lipschitz functions are uniformly continuous they have a unique extension to the completion $\overline{X}$ of $X$. Since $\overline{X}$ is compact, Arzel\`{a}--Ascoli applies.} Assume that $T$ is not uniformly pseudolocal. Then there would be $R, L > 0$ and $\varepsilon > 0$, so that for all $N > 0$ we would have an $f_N \in \LLip_R(X)$ such that for all rank-$N$ operators $k$ we have $\|[T, \rho(f_N)] - k\| \ge \varepsilon$. Since $\LLip_R(X)$ is relatively compact, the sequence $f_N$ has an accumulation point $f_\infty \in \LLip_R(X)$. Then we have $\|[T, \rho(f_\infty)] - k\| \ge \varepsilon / 2$ for all finite rank operators $k$, which is a contradiction.

The proofs that $T^2 - 1$ and $T - T^\ast$ are uniformly locally compact are analogous.
\end{proof}

A collection $(H, \rho, T_t)$ of uniform Fredholm modules is called an \emph{operator homotopy} if $t \mapsto T_t \in \IB(H)$ is norm continuous. As in the non-uniform case, we have an analogous lemma about compact perturbations:

\begin{lem}[Compact perturbations, {\cite[Lemma 2.16]{spakula_uniform_k_homology}}]
\label{lem:compact_perturbations}
Let $(H, \rho, T)$ be a uniform Fredholm module and $K \in \IB(H)$ a uniformly locally compact operator.

Then $(H, \rho, T)$ and $(H, \rho, T + K)$ are operator homotopic.
\end{lem}

The definition of uniform $K$-homology now proceeds as the one for usual $K$-homology:

\begin{defn}[Uniform $K$-homology, {\cite[Definition 2.13]{spakula_uniform_k_homology}}]
We define the \emph{uniform $K$-homology group $K_{p}^u(X)$} of a locally compact and separable metric space $X$ to be the abelian group generated by unitary equivalence classes of $p$-multigraded uniform Fredholm modules with the relations:
\begin{itemize}
\item if $x$ and $y$ are operator homotopic, then $[x] = [y]$, and
\item $[x] + [y] = [x \oplus y]$,
\end{itemize}
where $x$ and $y$ are $p$-multigraded uniform Fredholm modules.
\end{defn}

All the basic properties of usual $K$-homology do also hold for uniform $K$-homology (e.g., that degenerate uniform Fredholm modules represent the zero class, that we have formal $2$-periodicity $K_{p}^u(X) \cong K_{p+2}^u(X)$ for all $p \ge -1$, etc.).

For discussing functoriality of uniform $K$-homology we need the following definition:

\begin{defn}[Uniformly cobounded maps, {\cite[Definition 2.15]{spakula_uniform_k_homology}}]
Let us call a map $g\colon X \to Y$ with the property
\[\sup_{y \in Y} \diam (g^{-1}(B_r(y))) < \infty \text{ for all }r > 0\]
\emph{uniformly cobounded}\footnote{Block and Weinberger call this property \emph{effectively proper} in \cite{block_weinberger_1}. The author called it \emph{uniformly proper} in his thesis \cite{engel_phd}.}.

Note that if $X$ is proper, then every uniformly cobounded map is proper (i.e., preimages of compact subsets are compact).
\end{defn}

The following lemma about functoriality of uniform $K$-homology was proved by \Spakula (see the paragraph directly after \cite[Definition 2.15]{spakula_uniform_k_homology}).

\begin{lem}
Uniform $K$-homology is functorial with respect to uniformly cobounded, proper Lipschitz maps, i.e., if $g\colon X \to Y$ is uniformly cobounded, proper and Lipschitz, then it induces maps $g_\ast\colon K_\ast^u(X) \to K_\ast^u(Y)$ on uniform $K$-homology via
\[g_\ast [(H, \rho, T)] := [(H, \rho \circ g^\ast, T)],\]
where $g^\ast\colon C_0(Y) \to C_0(X)$, $f \mapsto f \circ g$ is the by $g$ induced map on functions.
\end{lem}

Recall that $K$-homology may be normalized in various ways, i.e., we may assume that the Fredholm modules have a certain form or a certain property and that this holds also for all homotopies.

Combining Lemmas 4.5 and 4.6 and Proposition 4.9 from \cite{spakula_uniform_k_homology}, we get the following:

\begin{lem}\label{lem:normalization_involutive}
We can normalize uniform $K$-homology $K_\ast^u(X)$ to involutive modules.\footnote{Recall that a Fredholm module $(H, \rho, T)$ is called involutive if $T = T^\ast$, $\|T\| \le 1$ and $T^2 = 1$.}
\end{lem}

The proof of the following Lemma \ref{lem:normalization_non-degenerate} in the non-uniform case may be found in, e.g., \cite[Lemma 8.3.8]{higson_roe}. The proof in the uniform case is analogous and the arguments similar to the ones in the proofs of \cite[Lemmas 4.5 \& 4.6]{spakula_uniform_k_homology}.

\begin{lem}\label{lem:normalization_non-degenerate}
Uniform $K$-homology $K_\ast^u(X)$ may be normalized to non-degenerate Fredholm modules, i.e., such that all occuring representations $\rho$ are non-degenerate\footnote{This means that $\rho(C_0(X)) H$ is dense in $H$.}.
\end{lem}

Note that in general we can not normalize uniform $K$-homology to be simultaneously involutive and non-degenerate, just as is the case for usual $K$-homology.

Later we will also have to normalize Fredholm modules to finite propagation. But this is not always possible if the underlying metric space $X$ is badly behaved. Therefore we get now to the definition of bounded geometry for metric spaces.

\begin{defn}[Coarsely bounded geometry]
\label{defn:coarsely_bounded_geometry}
Let $X$ be a metric space. We call a subset $\Gamma \subset X$ a \emph{quasi-lattice} if
\begin{itemize}
\item there is a $c > 0$ such that $B_c(\Gamma) = X$ (i.e., $\Gamma$ is \emph{coarsely dense}) and
\item for all $r > 0$ there is a $K_r > 0$ such that $\card(\Gamma \cap B_r(y)) \le K_r$ for all $y \in X$.
\end{itemize}
A metric space is said to have \emph{coarsely bounded geometry}\footnote{Note that most authors call this property just ``bounded geometry''. But since later we will also have the notion of locally bounded geometry, we use for this one the term ``coarsely'' to distinguish them.} if it admits a quasi-lattice.
\end{defn}

Note that if we have a quasi-lattice $\Gamma \subset X$, then there also exists a uniformly discrete quasi-lattice $\Gamma^\prime \subset X$. The proof of this is an easy application of the Lemma of Zorn: given an arbitrary $\delta > 0$ we look at the family $\mathcal{A}$ of all subsets $A \subset \Gamma$ with $d(x,y) > \delta$ for all $x,y \in A$. These subsets are partially ordered under inclusion of sets and every totally ordered chain $A_1 \subset A_2 \subset \ldots \subset \Gamma$ has an upper bound given by the union $\bigcup_i A_i \in \mathcal{A}$. So the Lemma of Zorn provides us with a maximal element $\Gamma^\prime \in \mathcal{A}$. That $\Gamma^\prime$ is a quasi-lattice follows from its maximality.

\begin{examples}\label{ex:coarsely_bounded_geometry}
Every Riemannian manifold $M$ of bounded geometry\footnote{That is to say, the injectivity radius of $M$ is uniformly positive and the curvature tensor and all its derivatives are bounded in sup-norm.} is a metric space of coarsely bounded geometry: any maximal set $\Gamma \subset M$ of points which are at least a fixed distance apart (i.e., there is an $\varepsilon > 0$ such that $d(x, y) \ge \varepsilon$ for all $x \not= y \in \Gamma$) will do the job. We can get such a maximal set by invoking Zorn's lemma. Note that a manifold of bounded geometry will also have locally bounded geometry (this notion will be defined further below), so no confusion can arise by not distinguishing between ``coarsely'' and ``locally'' bounded geometry in the terminology for manifolds.

If $(X,d)$ is an arbitrary metric space that is bounded, i.e., $d(x,x^\prime) < D$ for all $x, x^\prime \in X$ and some $D$, then \emph{any} finite subset of $X$ will constitute a quasi-lattice.

Let $K$ be a simplicial complex of bounded geometry\footnote{That is, the number of simplices in the link of each vertex is uniformly bounded.}. Equipping $K$ with the metric derived from barycentric coordinates the subset of all vertices of the complex $K$ becomes a quasi-lattice in $K$.
\end{examples}

If $X$ has coarsely bounded geometry it will be crucial for us that we can normalize uniform $K$-homology to uniform finite propagation, i.e., such there is an $R > 0$ depending only on $X$ such that every uniform Fredholm module has propagation at most $R$\footnote{This means $\rho(f) T \rho(g) = 0$ if $d(\supp f, \supp g) > R$.}. This was proved by \Spakula in \cite[Proposition 7.4]{spakula_uniform_k_homology}. Note that it is in general not possible to make this common propagation $R$ arbitrarily small. Furthermore, we can combine the normalization to finite propagation with the other normalizations.

\begin{prop}[{\cite[Section 7]{spakula_uniform_k_homology}}]\label{prop:normalization_finite_prop_speed}
If $X$ has coarsely bounded geometry, then there is an $R > 0$ depending only on $X$ such that uniform $K$-homology may be normalized to uniform Fredholm modules that have propagation at most $R$.

Furthermore, we can additionally normalize them to either involutive modules or to non-degenerate ones.
\end{prop}

Having discussed the normalization to finite propagation modules, we can now compute an easy but important example:

\begin{lem}\label{lem:uniform_k_hom_discrete_space}
Let $Y$ be a uniformly discrete, proper metric space of coarsely bounded geometry. Then $K_0^u(Y)$ is isomorphic to the group $\ell_\IZ^\infty(Y)$ of all bounded, integer-valued sequences indexed by $Y$, and $K_1^u(Y) = 0$.
\end{lem}

\begin{proof}
We use Proposition \ref{prop:normalization_finite_prop_speed} to normalize uniform $K$-homology to operators of finite propagation, i.e., there is an $R > 0$ such that every uniform Fredholm module over $Y$ may be represented by a module $(H, \rho, T)$ where $T$ has propagation no more than $R$ and all homotopies may be also represented by homotopies where the operators have propagation at most~$R$.

Going into the proof of Proposition \ref{prop:normalization_finite_prop_speed}, we see that in our case of a uniformly discrete metric space $Y$ we may choose~$R$ less than the least distance between two different points of $Y$, i.e., $0 < R < \inf_{x \not= y \in Y} d(x,y)$. Given now a module $(H, \rho, T)$ where $T$ has propagation at most this~$R$, the operator $T$ decomposes as a direct sum $T = \bigoplus_{y \in Y} T_y$ with $T_y \colon H_y \to H_y$. The Hilbert space $H_y$ is defined as $H_y := \rho(\chi_y) H$, where $\chi_y$ is the characteristic function of the single point $y \in Y$. Note that $\chi_y$ is a continuous function since the space $Y$ is discrete. Hence $(H, \rho, T) = \bigoplus (H_y, \rho_y, T_y)$ with $\rho_y\colon C_0(Y) \to \IB(H_y)$, $f \mapsto \rho(\chi_y) \rho(f) \rho(\chi_y)$. Now each $(H_y, \rho_y, T_y)$ is a Fredholm module over the point $y$ and so we get a map
\[K_\ast^u(Y) \to \prod_{y \in Y} K_\ast^u(y).\]
Note that we need that the homotopies also all have propagation at most~$R$ so that the above defined decomposition of a uniform Fredholm module descends to the level of uniform $K$-homology.

Since a point $y$ is for itself a compact space, we have $K_\ast^u(y) = K_\ast(y)$, and the latter group is isomorphic to $\IZ$ for $\ast = 0$ and it is $0$ for $\ast = 1$. Since the above map $K_\ast^u(Y) \to \prod_{y \in Y} K_\ast^u(y)$ is injective, we immediately conclude $K_1^u(Y) = 0$.

So it remains to show that the image of this map in the case $\ast = 0$ consists of the \emph{bounded} integer-valued sequences indexed by $Y$. But this follows from the uniformity condition in the definition of uniform $K$-homology: the isomorphism $K_0(y) \cong \IZ$ is given by assigning a module $(H_y, \rho_y, T_y)$ the Fredholm index of $T$ (note that $T_y$ is a Fredholm operator since $(H_y, \rho_y, T_y)$ is a module over a single point). Now since $(H, \rho, T) = \bigoplus (H_y, \rho_y, T_y)$ is a \emph{uniform} Fredholm module, we may conclude that the Fredholm indices of the single operators $T_y$ are bounded with respect to $y$.
\end{proof}

\subsection{Differences to \v{S}pakula's version}\label{sec_changes_defn}
We will discuss now the differences between our version of uniform $K$-homology and \v{S}pakula's version from his Ph.D.\ thesis \cite{spakula_thesis}, resp., his publication \cite{spakula_uniform_k_homology}.

Firstly, our definition of uniform $K$-homology is based on multigraded Fredholm modules and we therefore have groups $K_p^\ast(X)$ for all $p \ge -1$, but \Spakula only defined $K_0^u$ and $K_1^u$. This is not a real restriction since uniform $K$-homology has, analogously as usual $K$-homology, a formal $2$-periodicity. We mention this since if the reader wants to look up the original reference \cite{spakula_thesis} and \cite{spakula_uniform_k_homology}, he has to keep in mind that we work with multigraded modules, but \Spakula does not.

Secondly, \Spakula gives the definition of uniform $K$-homology only for proper\footnote{That means that all closed balls are compact.} metric spaces since certain results of him (Sections 8-9 in \cite{spakula_uniform_k_homology}) only work for such spaces. These results are all connected to the rough assembly map $\mu_u\colon K_\ast^u(X) \to K_\ast(C_u^\ast(Y))$, where $Y \subset X$ is a uniformly discrete quasi-lattice, and this is not surprising: the (uniform) Roe algebra only has on proper spaces nice properties (like its $K$-theory being a coarse invariant) and therefore we expect that results of uniform $K$-homology that connect to the uniform Roe algebra also should need the properness assumption. But we can see by looking into the proofs of \Spakula in all the other sections of \cite{spakula_uniform_k_homology} that all results except the ones in Sections 8-9 also hold for locally compact, separable metric spaces (without assumptions on completeness or properness). Note that this is a very crucial fact for us that uniform $K$-homology does also make sense for non-proper spaces since in the proof of \Poincare duality we will have to consider the uniform $K$-homology of open balls in $\IR^n$.

Thirdly, \v{S}pakula uses the notion ``$L$-continuous'' instead of ``$L$-Lipschitz'' for the definition of $\LLip_R(X)$ (which he also denotes by $C_{R,L}(X)$, i.e., we have also changed the notation), so that he gets slightly differently defined uniform Fredholm modules. But the author was not able to deduce Proposition \ref{prop:compact_space_every_module_uniform} with \v{S}pakula's definition, which is why we have changed it to ``$L$-Lipschitz'' (since the statement of Proposition \ref{prop:compact_space_every_module_uniform} is a very desirable one and, in fact, later we will need it crucially in the proof of \Poincare duality). \Spakula noted that for a geodesic metric space both notions ($L$-continuous and $L$-Lipschitz) coincide, i.e., for probably all spaces which one wants to consider ours and \v{S}pakula's choices coincide. But note that all the results of \Spakula do also hold with our definition of uniform Fredholm modules.

And last, let us get to the most crucial difference between the definitions: to define uniform $K$-homology \Spakula does not use operator homotopy as a relation but a certain weaker form of homotopy (\cite[Definition 2.11]{spakula_uniform_k_homology}). The reasons why we changed this are the following: firstly, the definition of usual $K$-homology uses operator homotopy and it seems desirable to have uniform $K$-homology to be analogously defined. Secondly, \v{S}pakula's proof of \cite[Proposition 4.9]{spakula_uniform_k_homology} seems not to be correct under his notion of homotopy, but it becomes correct if we use operator homotopy as a relation. So by changing the definition we ensure that \cite[Proposition 4.9]{spakula_uniform_k_homology} holds. And thirdly, we will prove in Section~\ref{sec:homotopy_invariance} that we get the same uniform $K$-homology groups if we impose weak homotopy (Definition~\ref{defn:weak_homotopy}) as a relation instead of operator homotopy. Though our notion of weak homotopies is different from \v{S}pakula's notion of homotopies, all the homotopies that he constructs in his paper \cite{spakula_uniform_k_homology} are also weak homotopies, i.e., all the results of him that rely on his notion of homotopy are also true with our definition.

To put it into a nutshell, we changed the definition of uniform $K$-homology in order to make the definition similar to one of usual $K$-homology and to correct \v{S}pakula's proof of \cite[Proposition 4.9]{spakula_uniform_k_homology}. It also seems to be easier to work with our version. Furthermore, all of his results do also hold in our definition. And last, we remark that his results, besides the ones in Sections 8-9 in \cite{spakula_uniform_k_homology}, also hold for non-proper, non-complete spaces.

\subsection{External product}

Now we get to one of the most important technical parts in this article: the construction of the external product for uniform $K$-homology. Its main application will be to deduce homotopy invariance of uniform $K$-homology.

Note that we can construct the product only if the involved metric spaces have jointly bounded geometry (which we will define in a moment). Note that both major classes of spaces on which we want to apply our theory, namely manifolds and simplicial complexes of bounded geometry, do have jointly bounded geometry.

\begin{defn}[Locally bounded geometry, {\cite[Definition 3.1]{spakula_universal_rep}}]\label{defn:locally_bounded_geometry}
A metric space $X$ has \emph{locally bounded geometry}, if it admits a countable Borel decomposition $X = \cup X_i$ such that
\begin{itemize}
\item each $X_i$ has non-empty interior,
\item each $X_i$ is totally bounded, and
\item for all $\varepsilon > 0$ there is an $N > 0$ such that for every $X_i$ there exists an $\varepsilon$-net in $X_i$ of cardinality at most $N$.
\end{itemize}

Note that \Spakula demands in his definition of ``locally bounded geometry'' that the closure of each $X_i$ is compact instead of the total boundedness of them. The reason for this is that he considers only proper spaces, whereas we need a more general notion to encompass also non-complete spaces.
\end{defn}

\begin{defn}[Jointly bounded geometry]\label{defn:jointly_bounded_geometry}
A metric space $X$ has \emph{jointly coarsely and locally bounded geometry}, if
\begin{itemize}
\item it admits a countable Borel decomposition $X = \cup X_i$ satisfying all the properties of the above Definition \ref{defn:locally_bounded_geometry} of locally bounded geometry,
\item it admits a quasi-lattice $\Gamma \subset X$ (i.e., $X$ has coarsely bounded geometry), and
\item for all $r > 0$ we have $\sup_{y \in \Gamma} \card \{i \ | \ B_r(y) \cap X_i \not= \emptyset\} < \infty$.
\end{itemize}
The last property ensures that there is an upper bound on the number of subsets $X_i$ that intersect any ball of radius $r > 0$ in $X$.
\end{defn}

\begin{examples}
Recall from Examples \ref{ex:coarsely_bounded_geometry} that manifolds of bounded geometry and simplicial complexes of bounded geometry (i.e., the number of simplices in the link of each vertex is uniformly bounded) equipped with the metric derived from barycentric coordinates have coarsely bounded geometry. Now a moment of reflection reveals that they even have jointly bounded geometry.

In the next Figure \ref{fig:not_jointly_but_others} we give an example of a space $X$ having coarsely and locally bounded geometry, but where the quasi-lattice $\Gamma$ and the Borel decomposition $X = \cup X_i$ are not compatible with each other, i.e., they do not provide $X$ with the structure of a space with locally bounded geometry.
\end{examples}

\begin{figure}[htbp]
\centering
\includegraphics[scale=0.5]{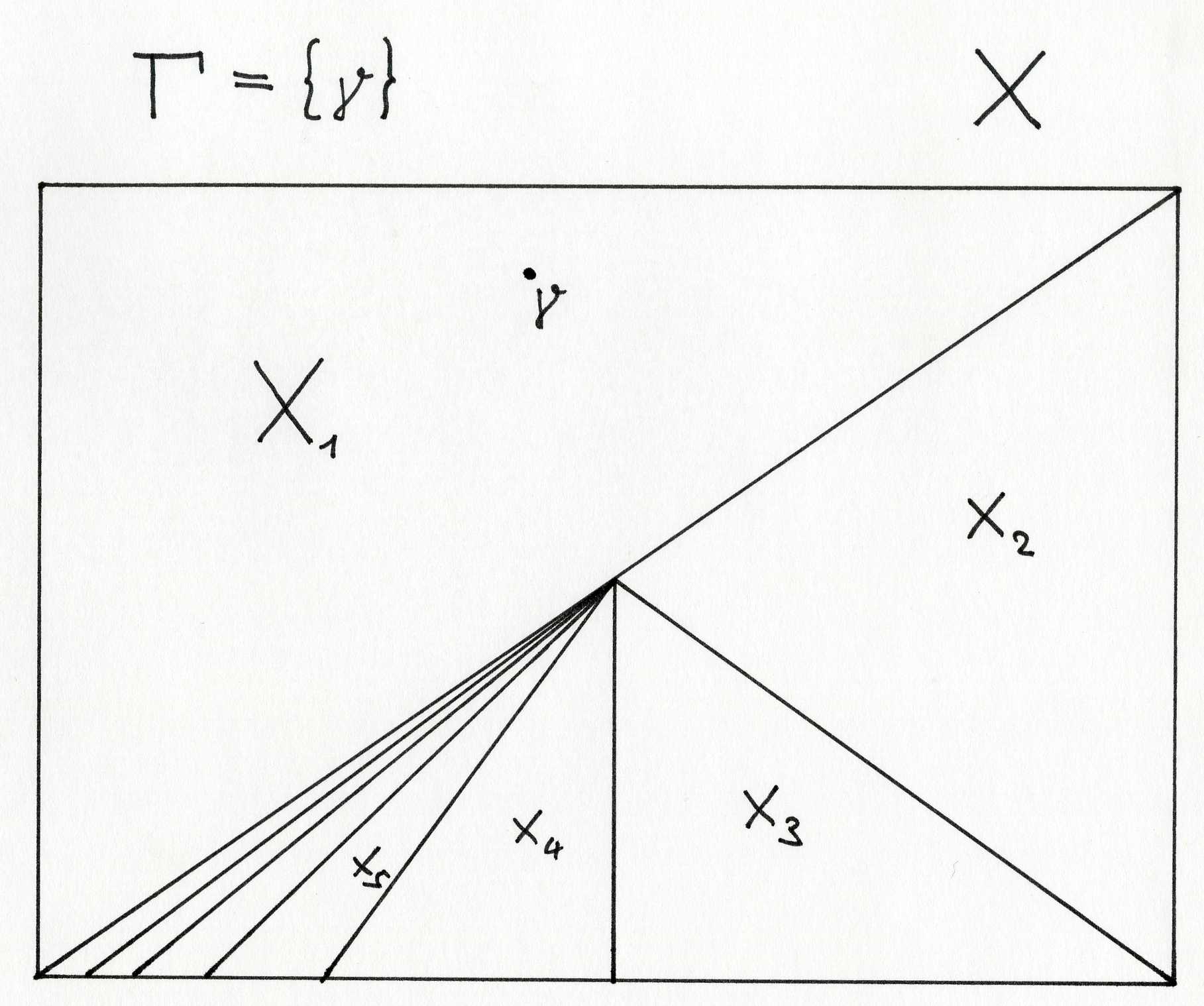}
\caption{Coarsely and locally bounded geometry, but they are not compatible.}
\label{fig:not_jointly_but_others}
\end{figure}

In our construction of the product for uniform $K$-homology we follow the presentation in \cite[Section 9.2]{higson_roe}, where the product is constructed for usual $K$-homology.

Let $X_1$ and $X_2$ be locally compact and separable metric spaces and both having jointly bounded geometry, $(H_1, \rho_1, T_1)$ a $p_1$-multigraded uniform Fredholm module over the space $X_1$ and $(H_2, \rho_2, T_2)$ a $p_2$-multigraded module over $X_2$, and both modules will be assumed to have finite propagation (see Proposition \ref{prop:normalization_finite_prop_speed}).

\begin{defn}[cf. {\cite[Definition 9.2.2]{higson_roe}}]
We define $\rho$ to be the tensor product representation of $C_0(X_1 \times X_2) \cong C_0(X_1) \otimes C_0(X_2)$ on $H := H_1 \hatotimes H_2$, i.e.,
\[\rho(f_1 \otimes f_2) = \rho_1(f_1) \hatotimes \rho_2(f_2) \in \IB(H_1) \hatotimes \IB(H_2)\]
and equip $H_1 \hatotimes H_2$ with the induced $(p_1 + p_2)$-multigrading\footnote{The graded tensor product $H_1 \hatotimes H_2$ is $(p_1 + p_2)$-multigraded if we let the multigrading operators $\epsilon_j$ of $H_1$ act on the tensor product as
\[\epsilon_j(v_1 \otimes v_2) := (-1)^{\deg(v_2)}\epsilon_j(v_1) \otimes v_2\]
for $1 \le j \le p_1$, and for $1 \le j \le p_2$ we let the multigrading operators $\epsilon_{p_1 + j}$ of $H_2$ act as
\[\epsilon_{p_1 + j}(v_1 \otimes v_2) := v_1 \otimes \epsilon_{p_1 + j}(v_2).\]}.

We say that a $(p_1 + p_2)$-multigraded uniform Fredholm module $(H, \rho, T)$ is \emph{aligned} with the modules $(H_1, \rho_1, T_1)$ and $(H_2, \rho_2, T_2)$, if
\begin{itemize}
\item $T$ has finite propagation,
\item for all $f \in C_0(X_1 \times X_2)$ the operators
\[\rho(f) \big( T (T_1 \hatotimes 1) + (T_1 \hatotimes 1) T \big) \rho(\bar f) \text{ and } \rho(f) \big( T (1 \hatotimes T_2) + (1 \hatotimes T_2) T \big) \rho(\bar f)\]
are positive modulo compact operators,\footnote{That is to say, they are positive in the Calkin algebra $\IB(H) / \IK(H)$.} and
\item for all $f\in C_0(X_1 \times X_2)$ the operator $\rho(f) T$ derives $\IK(H_1) \hatotimes \IB(H_2)$, i.e.,
\begin{equation}
\label{eq_derives}
[\rho(f) T, \IK(H_1) \hatotimes \IB(H_2)] \subset \IK(H_1) \hatotimes \IB(H_2).
\end{equation}
\end{itemize}
Since both $H$ and $\rho$ are uniquely determined from $H_1$, $\rho_1$, $H_2$ and $\rho_2$, we will often just say that \emph{$T$ is aligned with $T_1$ and $T_2$}.
\end{defn}

Our major technical lemma is the following one. It is a uniform version of Kasparov's Technical Lemma, which is suitable for our needs.

\begin{lem}\label{lem:construction_partition_unity}
Let $X_1$ and $X_2$ be locally compact and separable metric spaces that have jointly coarsely and locally bounded geometry.

Then there exist commuting, even, multigraded, positive operators $N_1$, $N_2$ of finite propagation on $H := H_1 \hatotimes H_2$ with $N_1^2 + N_2^2 = 1$ and the following properties:

\begin{enumerate}
\item $N_1 \cdot \big\{ (T_1^2 - 1) \rho_1(f) \hatotimes 1 \ | \ f \in \LLip_{R^\prime}(X_1) \big\} \subset \IK(H_1 \hatotimes H_2)$ is uniformly approximable for all $R^\prime, L > 0$ and analogously for $(T^\ast_1 - T_1)\rho_1(f)$ and for $[T_1, \rho_1(f)]$ instead of $(T_1^2 - 1) \rho_1(f)$,
\item $N_2 \cdot \big\{ 1 \hatotimes (T_2^2 - 1) \rho_2(f) \ | \ f \in \LLip_{R^\prime}(X_2) \big\} \subset \IK(H_1 \hatotimes H_2)$ is uniformly approximable for all $R^\prime, L > 0$ and analogously for $(T_2^\ast - T_2) \rho_2(f)$ and for $[T_2, \rho_2(f)]$ instead of $(T_2^2 - 1) \rho_2(f)$,
\item $\{[N_i, T_1 \hatotimes 1]\rho(f), [N_i, 1 \hatotimes T_2]\rho(f) \ | \ f \in \LLip_{R^\prime}(X_1 \times X_2)\}$ is uniformly approximable for all $R^\prime, L > 0$ and both $i=1,2$,
\item $\big\{ [N_i, \rho(f \otimes 1)], [N_i, \rho(1 \otimes g)] \ | \ f \in \LLip_{R^\prime}(X_1), g \in \LLip_{R^\prime}(X_2) \big\}$ is uniformly approximable for all $R^\prime, L > 0$ and both $i = 1,2$, and
\item both $N_1$ and $N_2$ derive $\IK(H_1) \hatotimes \IB(H_2)$.\footnote{see \eqref{eq_derives}}
\end{enumerate}
\end{lem}

\begin{proof}
Due to the jointly bounded geometry there is a countable Borel decomposition $\{X_{1,i}\}$ of $X_1$ such that each $X_{1,i}$ has non-empty interior, the completions $\{\overline{X_{1,i}}\}$ form an admissible class\footnote{This means that for every $\varepsilon > 0$ there is an $N > 0$ such that in every $\overline{X_{1,i}}$ exists an $\varepsilon$-net of cardinality at most $N$.} of compact metric spaces and for each $R > 0$ we have
\begin{equation}
\label{eq:bound_jointly_bounded_geom}
\sup_i \card \{j \ | \ B_R(X_{1,i}) \cap X_{1,j} \not= \emptyset\} < \infty.
\end{equation}

The completions of the $1$-balls $B_1(X_{1,i})$ are also an admissible class of compact metric spaces and the collection of these open balls forms a uniformly locally finite open cover of $X_1$. We may find a partition of unity $\varphi_{1,i}$ subordinate to the cover $\{B_1(X_{1,i})\}$ such that every function $\varphi_{1,i}$ is $L_0$-Lipschitz for a fixed $L_0 > 0$ (but we will probably have to enlarge the value of $L_0$ a bit in a moment). The same holds also for a countable Borel decomposition $\{X_{2,i}\}$ of $X_2$ and we choose a partition of unity $\varphi_{2,i}$ subordinate to the cover $\{B_1(X_{2,i})\}$ such that every function $\varphi_{2,i}$ is also $L_0$-Lipschitz (by possibly enlargening $L_0$ so that we have the same Lipschitz constant for both partitions of unity).

Since $\{\overline{B_1(X_{1,i})}\}$ is an admissible class of compact metric spaces, we have for each $\varepsilon > 0$ and $L > 0$ a bound independent of $i$ on the number of functions from
\[\varphi_{1,i} \cdot \LLip_c(X_1) := \{ \varphi_{1,i} \cdot f \ | \ f\text{ is }L\text{-Lipschitz, compactly supported and }\|f\|_\infty \le 1\}\]
to form an $\varepsilon$-net in $\varphi_{1,i} \cdot \LLip_c(X_1)$, and analogously for $X_2$ (this can be proved by a similar construction as the one from \cite[Lemma 2.4]{spakula_universal_rep}). We denote this upper bound by $C_{\varepsilon, L}$.

Now for each $N \in \IN$ and $i \in \IN$ we choose $C_{1/N,N}$ functions $\{f_k^{i,N}\}_{k=1, \ldots, C_{1/N, N}}$ from $\varphi_{1,i} \cdot N\text{-}\operatorname{Lip}_c(X_{1,i})$ constituting an $1/N$-net.\footnote{If we need less functions to get an $1/N$-net, we still choose $C_{1/N,N}$ of them. This makes things easier for us to write down.} Analogously we choose $C_{1/N, N}$ functions $\{g_k^{i,N}\}_{k=1, \ldots, C_{1/N, N}}$ from $\varphi_{2,i} \cdot N\text{-}\operatorname{Lip}_c(X_{2,i})$ that are $1/N$-nets.

We choose a sequence $\{u_n \hatotimes 1\} \subset \IB(H_1) \hatotimes \IB(H_2)$ of operators in the following way: $u_n$ will be a projection operator onto a subspace $U_n$ of $H_1$. To define this subspace, we first consider the operators
\begin{equation}\label{eq:operators_X_1_to_approximate}
(T_1^2 - 1)\rho_1(f), \ (T_1 - T_1^\ast)\rho_1(f), \text{ and } [T_1, \rho_1(f)]
\end{equation}
for suitable functions $f \in C_0(X_1)$ that we will choose in a moment. These operators are elements of $\IK(H_1)$ since $(H_1, \rho_1, T_1)$ is a Fredholm module. So up to an error of $2^{-n}$ they are of finite rank and the span $V_n$ of the images of these finite rank operators will be the building block for the subspace $U_n$ on which the operator $u_n$ projects\footnote{This finite rank operators are of course not unique. Recall that every compact operator on a Hilbert space $H$ may be represented in the form $\sum_{n \ge 1} \lambda_n \langle f_n, - \rangle g_n$, where the values $\lambda_n$ are the singular values of the operator and $\{f_n\}$, $\{g_n\}$ are orthonormal (though not necessarily complete) families in $H$ (but contrary to the $\lambda_n$ they are not unique). Now we choose our finite rank operator to be the operator given by the same sum, but only with the $\lambda_n$ satisfying $\lambda_n \ge 2^{-n}$.} (i.e., we will say in a moment how to enlarge $V_n$ in order to get $U_n$). We choose the functions $f \in C_0(X_1)$ as all the functions from the set $\bigcup \{f_k^{i,N}\}_{k=1, \ldots, C_{1/N,N}}$, where the union ranges over all $i \in \IN$ and $1 \le N \le n$. Note that since the Fredholm module $(H_1, \rho_1, T_1)$ is uniform, the rank of the finite rank operators approximating \eqref{eq:operators_X_1_to_approximate} up to an error of $2^{-n}$ is bounded from above with a bound that depends only on $N$ and $n$, but not on $i$ nor $k$. Since we will have $V_n \subset U_n$, we can already give the first estimate that we will need later:
\begin{equation}\label{eq:Kasparov_estimate_383}
\|(u_n \hatotimes 1)(x \hatotimes 1) - (x \hatotimes 1)\| < 2^{-n},
\end{equation}
where $x$ is one of the operators from \eqref{eq:operators_X_1_to_approximate} for all $f_k^{i,N}$ with $1 \le N \le n$.\footnote{Actually, to have this estimate we would need that $x$ is self-adjoint. We can pass from $x$ to $\tfrac{1}{2}(x + x^\ast)$ and $\tfrac{1}{2i}(x - x^\ast)$, do all the constructions with these self-adjoint operators and get the needed estimates for them, and then we get the same estimates for $x$ but with an additional factor of $2$.} Moreover, denoting by $\chi_{1,i}$ the characteristic function of $B_1(X_{1,i})$, then $\rho_1(\chi_{1,i}) \cdot V_n$ is a subspace of $H_1$ of finite dimension that is bounded independently of $i$.\footnote{We have used here the fact that we may uniquely extend any representation of $C_0(Z)$ to one of the bounded Borel functions $B_b(Z)$ on a space $Z$.} The reason for this is because $T_1$ has finite propagation and the number of functions $f_k^{i,N}$ for fixed $N$ is bounded independently of $i$. For all $n$ we also have $V_n \subset V_{n+1}$ and that the projection operator onto $V_n$ has finite propagation which is bounded independently of $n$.

For each $n \in \IN$ we partition $\chi_{1,i}$ for all $i \in \IN$ into disjoint characteristic functions $\chi_{1,i} = \sum_{j=1}^{J_n} \chi_{1,i}^{j,n}$ such that we may write each function $f_k^{i,N}$ for all $i \in \IN$, $1 \le N \le n$ and $k = 1, \ldots, C_{1/N, N}$ up to an error of $2^{-n-1}$ as a sum $f_k^{i,N} = \sum_{j=1}^{J_n} \alpha_k^{i,N}(j,n) \cdot \chi_{1,i}^{j,n}$ for suitable constants $\alpha_k^{i,N}(j,n)$. Note that since $X_1$ has jointly coarsely and locally bounded geometry, we can choose the upper bounds $J_n$ such that they do not depend on $i$. Now we can finally set $U_n$ as the linear span of $V_n$ and $\rho_1(\chi_{1,i}^{j,n}) \cdot V_n$ for all $i \in \IN$ and $1 \le j \le J_n$. Note that $\rho_1(\chi_{1,i}) \cdot U_n$ is a subspace of $H_1$ of finite dimension that is bounded independently of $i$, that we may choose the characteristic functions $\chi_{1,i}^{j,n}$ such that we have $U_n \subset U_{n+1}$ (by possibly enlargening each $J_n$), and that the projection operator $u_n$ onto $U_n$ has finite propagation which is bounded independently of $n$. Since we have $[u_n, \rho_1(\chi_{1,i}^{j,n})] = 0$ for all $i \in \IN$, $1 \le j \le J_n$ and all $n \in \IN$, we get our second crucial estimate:
\begin{equation}\label{eq:Kasparov_estimate_384}
\|[u_n \hatotimes 1, \rho_1(f_k^{i,N}) \hatotimes 1]\| < 2^{-n}
\end{equation}
for all $i \in \IN$, $k = 1, \ldots, C_{1/N, N}$, $1 \le N \le n$ and all $n \in \IN$.

By an argument similar to the proof of the existence of quasicentral approximate units, we may conclude that for each $n \in \IN$ there exists a finite convex combination $\nu_n$ of the elements $\{u_n, u_{n+1}, \ldots\}$ such that
\begin{equation}\label{eq:Kasparov_estimate_384_2}
\|[\nu_n \hatotimes 1, T_1 \hatotimes 1]\| < 2^{-n} \text{, } \|[\nu_n \hatotimes 1, \epsilon_1 \hatotimes \epsilon_2]\| < 2^{-n} \text{ and } \|[\nu_n \hatotimes 1, \epsilon^j]\| < 2^{-n}
\end{equation}
for all $n \in \IN$, where $\epsilon_1 \hatotimes \epsilon_2$ is the grading operator of $H_1 \hatotimes H_2$ and $\epsilon^j$, $1 \le j \le p_1 + p_2$, are the multigrading operators of $H_1 \hatotimes H_2$. Note that the Estimates \eqref{eq:Kasparov_estimate_383} and \eqref{eq:Kasparov_estimate_384} also hold for $\nu_n$. Note furthermore that we can arrange that the maximal index occuring in the finite convex combination for $\nu_n$ is increasing in $n$.

Now we will construct a sequence $w_n \in \IB(H_1) \hatotimes \IB(H_2)$ with suitable properties. We have that $\nu_n$ is a finite convex combination of the elements $\{u_n, u_{n+1}, \ldots\}$. So for $n \in \IN$ we let $m_n$ denote the maximal occuring index in that combination. Furthermore, we let the projections $p_n \in \IB(H_2)$ be analogously defined as $u_n$, where we consider now the operators
\begin{equation}\label{eq:operators_X_2_to_approximate}
(T_2^2 - 1)\rho_2(g), \ (T_2 - T_2^\ast)\rho_2(g), \text{ and } [T_2, \rho_2(g)]
\end{equation}
for the analogous sets of functions $\bigcup \{g_k^{i,N}\}_{k=1, \ldots, C_{1/N,N}}$ depending on $n \in \IN$. Then we define $w_{n-1} := u_{m_n} \hatotimes  p_n$\footnote{The index is shifted by one so that we get the Estimates \eqref{eq:Kasparov_estimate_386}--\eqref{eq:Kasparov_estimate_388} with $2^{-n}$ and not with $2^{-n+1}$; though this is not necessary for the argument.} and get for all $n \in \IN$ the following:
\begin{equation}
\label{eq:Kasparov_estimate_385_-1}
w_n (\nu_n \hatotimes 1) (1 \hatotimes p_n) = (\nu_n \hatotimes 1) (1 \hatotimes p_n)
\end{equation}
and
\begin{align}
\label{eq:Kasparov_estimate_386}
\| [ w_n, x \hatotimes 1 ] \| & < 2^{-n}\\
\label{eq:Kasparov_estimate_387}
\| [ w_n, 1 \hatotimes y ] \| & < 2^{-n}\\
\label{eq:Kasparov_estimate_388}
\| [ w_n, \rho(f_k^{i,N} \otimes g_k^{i,N}) ] \| & < 2^{-n}
\end{align}
for all $i \in \IN$, $1 \le N \le n$ and $k = 1, \ldots, C_{1/N,N}$, where $x$ is one of the operators from \eqref{eq:operators_X_1_to_approximate} for all $f_k^{i,N}$ and $y$ is one of the operators from \eqref{eq:operators_X_2_to_approximate} for all $g_k^{i,N}$.

Let now $d_n := (w_n - w_{n-1})^{1/2}$. With a suitable index shift we can arrange that firstly, the Estimates \eqref{eq:Kasparov_estimate_386}--\eqref{eq:Kasparov_estimate_388} also hold for $d_n$ instead of $w_n$,\footnote{see \cite[Exercise 3.9.6]{higson_roe}} and that secondly, using Equation \eqref{eq:Kasparov_estimate_385_-1},
\begin{equation}
\label{eq:Kasparov_estimate_385}
\| d_n (\nu_n \hatotimes 1) y \| < 2^{-n},
\end{equation}
where $y$ is again one of the operators from \eqref{eq:operators_X_2_to_approximate} for all $g_k^{i,N}$ and $1 \le N \le n$.

Now as in the same way as we constructed $\nu_n$ out of the $u_n$s, we construct $\delta_n$ as a finite convex combination of the elements $\{d_n, d_{n+1}, \ldots\}$ such that
\begin{equation*}
\|[\delta_n, T_1 \hatotimes 1]\| < 2^{-n} \text{, } \|[\delta_n, 1 \hatotimes T_2]\| < 2^{-n} \text{, } \|[\delta_n, \epsilon_1 \hatotimes \epsilon_2]\| < 2^{-n} \text{ and } \|[\delta_n, \epsilon^j]\| < 2^{-n},\notag
\end{equation*}
where $\epsilon_1 \hatotimes \epsilon_2$ is the grading operator of $H_1 \hatotimes H_2$ and $\epsilon^j$ for $1 \le j \le p_1 + p_2$ are the multigrading operators of $H_1 \hatotimes H_2$. Clearly, all the Estimates \eqref{eq:Kasparov_estimate_386}--\eqref{eq:Kasparov_estimate_385} also hold 
for the operators $\delta_n$.

Define $X := \sum \delta_n \nu_n \delta_n$. It is a positive operator of finite propagation and fulfills the Points 2--4 that $N_2$ should have. The arguments for this are analogous to the ones given at the end of the proof of \cite[Kasparov's Technical Theorem 3.8.1]{higson_roe}, but we have to use all the uniform approximations that we additionally have (to use them, we have to cut functions $f \in \LLip_{R^\prime}(X_1)$ down to the single ``parts'' $X_{1,i}$ of $X_1$ by using the partition of unity $\{\varphi_{1,i}\}$ that we have chosen at the beginning of this proof, and analogously for $X_2$). Furthermore, the operator $1-X$ fulfills the desired Points 1, 3 and 4 that $N_1$ should fulfill. That both $X$ and $1-X$ derive $\IK(H_1) \hatotimes \IB(H_2)$ is clear via construction. Since $X$ commutes modulo compact operators with the grading and multigrading operators, we can average it over them so that it becomes an even and multigraded operator and $X$ and $1-X$ still have all the above mentioned properties.

Finally, we set $N_1 := (1-X)^{1/2}$ and $N_2 := X^{1/2}$.
\end{proof}

Now we will use this technical lemma to construct the external product and to show that it is well-defined on the level of uniform $K$-homology.

\begin{prop}\label{prop:external_prod_exists}
Let $X_1$ and $X_2$ be locally compact and separable metric spaces that have jointly coarsely and locally bounded geometry.

Then there exists a $(p_1 + p_2)$-multigraded uniform Fredholm module $(H, \rho, T)$ which is aligned with the modules $(H_1, \rho_1, T_1)$ and $(H_2, \rho_2, T_2)$.

Furthermore, any two such aligned Fredholm modules are operator homotopic and this operator homotopy class is uniquely determined by the operator homotopy classes of $(H_1, \rho_1, T_1)$ and $(H_2, \rho_2, T_2)$.
\end{prop}

\begin{proof}
We invoke the above Lemma \ref{lem:construction_partition_unity} to get operators $N_1$ and $N_2$ and then set
\[T := N_1(T_1 \hatotimes 1) + N_2(1 \hatotimes T_2).\]

To deduce that $(H, \rho, T)$ is a uniform Fredholm module, we have to use the following facts (additionally to the ones that $N_1$ and $N_2$ have): that $T_1$ and $T_2$ have finite propagation and are odd (we need that $(T_1 \hatotimes 1)(1 \hatotimes T_2) + (1 \hatotimes T_2)(T_1 \hatotimes 1) = 0$). To deduce that it is a multigraded module, we need that we constructed $N_1$ and $N_2$ as even and multigraded operators on $H$.

It is easily seen that for all $f \in C_0(X_1 \times X_2)$
\[\rho(f) \big( T (T_1 \hatotimes 1) + (T_1 \hatotimes 1) T \big) \rho(\bar f) \text{ and } \rho(f) \big( T (1 \hatotimes T_2) + (1 \hatotimes T_2) T \big) \rho(\bar f)\]
are positive modulo compact operators and that $\rho(f)T$ derives $\IK(H_1) \hatotimes \IB(H_2)$, i.e., we conclude that $T$ is aligned with $T_1$ and $T_2$.

Since all four operators $T_1$, $T_2$, $N_1$ and $N_2$ have finite propagation, $T$ has also finite propagation.

Suppose that $T^\prime$ is another operator aligned with $T_1$ and $T_2$. We construct again operators $N_1$ and $N_2$ using the above Lemma \ref{lem:construction_partition_unity}, but we additionally enforce
\[\|[w_n, \rho(f^{i,N}_k \otimes g_k^{i,N}) T^\prime]\| < 2^{-n}\]
analogously as we did it there to get Equation \eqref{eq:Kasparov_estimate_388}. So $N_1$ and $N_2$ will commute modulo compact operators with $\rho(f)T^\prime$ for all functions $f \in C_0(X_1 \times X_2)$. Again, we set $T := N_1(T_1 \hatotimes 1) + N_2(1 \hatotimes T_2)$. Since $N_1$ and $N_2$ commute modulo compacts with $\rho(f)T^\prime$ for all $f \in C_0(X_1 \times X_2)$ and since $T^\prime$ is aligned with $T_1$ and $T_2$, we conclude
\[\rho(f)(T T^\prime + T^\prime T) \rho(\bar f) \ge 0\]
modulo compact operators for all functions $f \in C_0(X_1 \times X_2)$. Using a uniform version of \cite[Proposition 8.3.16]{higson_roe} we conclude that $T$ and $T^\prime$ are operator homotopic via multigraded, uniform Fredholm modules. We conclude that every aligned module is operator homotopic to one of the form that we constructed above, i.e., to one of the form $N_1(T_1 \hatotimes 1) + N_2(1 \hatotimes T_2)$. But all such operators are homotopic to one another: they are determined by the operator $Y = N_2^2$ used in the proof of the above lemma and the set of all operators with the same properties as $Y$ is convex.

At last, suppose that one of the operators is varied by an operator homotopy, e.g., $T_1$ by $T_1(t)$. Then, in order to construct $N_1$ and $N_2$, we enforce in Equation \eqref{eq:Kasparov_estimate_384_2} instead of $\|[\nu_n \hatotimes 1, T_1 \hatotimes 1]\| < 2^{-n}$ the following one:
\[\|[\nu_n \hatotimes 1, T_1(j/n) \hatotimes 1]\| < 2^{-n}\]
for $0 \le j \le n$. Now we may define
\[T(t) := N_1(T_1(t) \hatotimes 1) + N_2(1 \hatotimes T_2),\]
i.e., we got operators $N_1$ and $N_2$ which are independent of $t$ but still have all the needed properties. This gives us the desired operator homotopy.
\end{proof}

\begin{defn}[External product]
The \emph{external product} of the multigraded uniform Fredholm modules $(H_1, \rho_1, T_1)$ and $(H_2, \rho_2, T_2)$ is a multigraded uniform Fredholm module $(H, \rho, T)$ which is aligned with $T_1$ and $T_2$. We will use the notation $T := T_1 \times T_2$.

By the above Proposition \ref{prop:external_prod_exists} we know that if the locally compact and separable metric spaces $X_1$ and $X_2$ both have jointly coarsely and locally bounded geometry, then the external product always exists, that it is well-defined up to operator homotopy and that it descends to a well-defined product on the level of uniform $K$-homology:
\[K_{p_1}^u(X_1) \times K_{p_2}^u(X_2) \to K_{p_1+p_2}^u(X_1 \times X_2)\]
for $p_1, p_2 \ge 0$. Furthermore, this product is bilinear.\footnote{To see this, suppose that, e.g., $T_1 = T_1^\prime \oplus T_1^{\prime \prime}$. Then it suffices to show that $T_1^\prime \times T_2 \oplus T_1^{\prime \prime} \times T_2$ is aligned with $T_1$ and $T_2$, which is not hard to do.}
\end{defn}

For the remaining products (i.e., the product of an ungraded and a multigraded module, resp., the product of two ungraded modules) we can appeal to the formal $2$-periodicity.

Associativity of the external product and the other important properties of it may be shown as in the non-uniform case. Let us summarize them in the following theorem:

\begin{thm}[External product for uniform $K$-homology]\label{thm:external_prod_homology}
Let $X_1$ and $X_2$ be locally compact and separable metric spaces of jointly bounded geometry\footnote{see Definition \ref{defn:jointly_bounded_geometry}}.

Then there exists an associative product
\[\times \colon K_{p_1}^u(X_1) \otimes K_{p_2}^u(X_2) \to K^u_{p_1 + p_2}(X_1 \times X_2)\]
for $p_1 , p_2 \ge -1$ with the following properties:
\begin{itemize}
\item for the flip map $\tau\colon X_1 \times X_2 \to X_2 \times X_1$ and all elements $[T_1] \in K_{p_1}^u(X_1)$ and $[T_2] \in K_{p_2}^u(X_2)$ we have
\[\tau_{\ast}[T_1 \times T_2] = (-1)^{p_1 p_2} [T_2 \times T_1],\]
\item we have for $g\colon Y \to Z$ a uniformly cobounded, proper Lipschitz map and elements $[T] \in K_{p_1}^u(X)$ and $[S] \in K_{p_2}^u(Y)$
\[(\id_X \operatorname{\times} g)_\ast [T \times S] = [T] \times g_\ast[S] \in K^u_{p_1 + p_2}(X \times Z),\]
and
\item denoting the generator of $K_0^u(\pt) \cong \IZ$ by $[1]$, we have
\[[T] \times [1] = [T] = [1] \times [T] \in K_\ast^u(X)\]
for all $[T] \in K_\ast^u(X)$.
\end{itemize}
\end{thm}

\subsection{Homotopy invariance}\label{sec:homotopy_invariance}

Let $X$ and $Y$ be locally compact, separable metric spaces with jointly bounded geometry and let $g_0, g_1 \colon X \to Y$ be uniformly cobounded, proper and Lipschitz maps which are homotopic in the following sense: there is a uniformly cobounded, proper and Lipschitz map $G\colon X \times [0,1] \to Y$ with $G(x,0) = g_0(x)$ and $G(x,1) = g_1(x)$ for all $x \in X$.

\begin{thm}\label{thm:homotopy_equivalence_k_hom}
If $g_0, g_1\colon X \to Y$ are homotopic in the above sense, then they induce the same maps $(g_0)_\ast = (g_1)_\ast \colon K_\ast^u(X) \to K_\ast^u(Y)$ on uniform $K$-homology.
\end{thm}

The proof of the above theorem is completely analogous to the non-uniform case and uses the external product. Furthermore, the above theorem is a special case of the following invariance of uniform $K$-homology under weak homotopies: given a uniform Fredholm module $(H, \rho, T)$ over $X$, the push-forward of it under $g_i$ is defined as $(H, \rho \circ g_i^\ast, T)$ and it is easily seen that these modules are weakly homotopic via the map $G$.

\begin{defn}[Weak homotopies]\label{defn:weak_homotopy}
Let a time-parametrized family of uniform Fredholm modules $(H, \rho_t, T_t)$ for $t \in [0,1]$ satisfy the following properties:
\begin{itemize}
\item the family $\rho_t$ is pointwise strong-$^\ast$ operator continuous, i.e., for all $f \in C_0(X)$ we get a path $\rho_t(f)$ in $\IB(H)$ that is continuous in the strong-$^\ast$ operator topology\footnote{Recall that if $H$ is a Hilbert space, then the \emph{strong-$^\ast$ operator topology} on $\IB(H)$ is generated by the family of seminorms $p_v(T) := \|Tv\| + \|T^\ast v\|$ for all $v \in H$, where $T \in \IB(H)$.},
\item the family $T_t$ is continuous in the strong-$^\ast$ operator topology on $\IB(H)$, i.e., for all $v \in H$ we get norm continuous paths $T_t(v)$ and $T_t^\ast(v)$ in $H$, and
\item for all $f \in C_0(X)$ the families of compact operators $[T_t, \rho_t(f)]$, $(T_t^2 - 1)\rho_t(f)$ and $(T_t-T_t^\ast)\rho_t(f)$ are norm continuous.
\end{itemize}
Then we call it a \emph{weak homotopy} between $(H, \rho_0, T_0)$ and $(H, \rho_1, T_1)$.
\end{defn}

\begin{rem}
If $\rho_t$ is pointwise norm continuous and $T_t$ is norm continuous, then the modules are weakly homotopic. So weak homotopy generalizes operator homotopy.
\end{rem}

\begin{thm}\label{thm:weak_homotopy_equivalence_K_hom}
Let $(H, \rho_0, T_0)$ and $(H, \rho_1, T_1)$ be weakly homotopic uniform Fredholm modules over a locally compact and separable metric space $X$ of jointly bounded geometry.

Then they define the same uniform $K$-homology class.
\end{thm}

\begin{proof}
Let our weakly homotopic family $(H, \rho_t, T_t)$ be parametrized by $t \in [0, 2\pi]$ so that our notation here will coincide with the one in the proof of \cite[Theorem 1 in §6]{kasparov_KK} that we mimic. Furthermore, we assume that $\rho_t$ and $T_t$ are constant in the intervals $[0, 2\pi/3]$ and $[4\pi / 3, 2\pi]$

We consider the graded Hilbert space $\IH := H \hatotimes (L^2[0,2\pi] \oplus L^2[0, 2\pi])$ (where the space $L^2[0,2\pi] \oplus L^2[0, 2\pi]$ is graded by interchanging the summands).

The family $T_t$ maps continuous paths $v_t$ in $H$ again to continuous paths $T_t(v_t)$: since the family $T_t$ is continuous in the strong-$^\ast$ operator topology and since it is defined on the compact interval $[0,1]$, we conclude with the uniform boundedness principle $\sup_t \|T_t\|_{op} < \infty$. Now if $t_n \to t$ is a convergent sequence, we get
\begin{align*}
\|T_{t_n} (v_{t_n}) - T_t (v_t)\| & \le \|T_{t_n} (v_{t_n}) - T_{t_n} (v_t)\| + \|T_{t_n} (v_t) - T_t (v_t)\|\\
& \le \underbrace{\|T_{t_n}\|_{op}}_{< \infty} \cdot \underbrace{\|v_{t_n} - v_t\|}_{\to 0} + \underbrace{\|(T_{t_n} - T_t)(v_t)\|}_{\to 0},
\end{align*}
where the second limit to $0$ holds due to the continuity of $T_t$ in the strong-$^\ast$ operator topology. So the family $T_t$ maps the dense subspace $H \otimes C[0,2\pi]$ of $H \otimes L^2[0,2\pi]$ into itself, and since it is norm bounded from above by $\sup_t \|T_t\|_{op} < \infty$, it defines a bounded operator on $H \otimes L^2[0,2\pi]$. We define an odd operator $\begin{pmatrix} 0 & T_t^\ast \\ T_t & 0\end{pmatrix}$ on $\IH$, which we also denote by $T_t$ (there should arise no confusion by using the same notation here).

Since $\rho_t(f)$ is strong-$^\ast$ continuous in $t$, we can analogously show that it maps continuous paths $v_t$ in $H$ again to continuous paths $\rho_t(f) (v_t)$, and it is norm bounded from above by $\|f\|_\infty$. because we have $\|\rho_t(f)\|_{op} \le \|f\|_\infty$ for all $t \in [0,1]$ since $\rho_t$ are representations of $C^\ast$-algebras. So $\rho_t(f)$ defines a bounded operator on $H \otimes L^2[0,2\pi]$ and we can get a representation $\rho_t \oplus \rho_t$ of $C_0(X)$ on $\IH$ by even operators, that we denote by the symbol $\rho_t$ (again, no confusion should arise by using the same notation).

We consider now the uniform Fredholm module
\[(\IH, \rho_t, N_1(T_t) + N_2(1 \hatotimes T(f)),\]
where $T(f)$ is defined as in the proof of \cite[Theorem 1 in §6]{kasparov_KK} (unfortunately, the overloading of the symbol ``$T$'' is unavoidable here). For the convenience of the reader, we will recall the definition of the operator $T(f)$ in a moment. That we may find a suitable partition of unity $N_1, N_2$ is due to the last bullet point in the definition of weak homotopies, and the construction of $N_1, N_2$ proceeds as in the end of the proof of our Proposition \ref{prop:external_prod_exists}.

To define $T(f)$, we first define an operator $d\colon L^2[0, 2\pi] \to L^2[0, 2\pi]$ using the basis $1, \ldots, \cos nx, \ldots, \sin nx, \ldots$ by the formulas
\[d(1) := 0 \text{, } d(\sin nx) := \cos nx \text{ and } d(\cos nx) := - \sin nx.\]
This operator $d$ is anti-selfadjoint, $d^2 + 1 \in \IK(L^2[0, 2\pi])$, and $d$ commutes modulo compact operators with multiplication by functions from $C[0, 2\pi]$. Let $f \in C[0, 2\pi]$ be a continuous, real-valued function with $|f(x)| \le 1$ for all $x \in [0, 2\pi]$, $f(0) = 1$ and $f(2\pi) = -1$. Then we set $T_1(f) := f - \sqrt{1 - f^2}\cdot d \in \IB(L^2[0, 2\pi])$. This operator $T_1(f)$ is Fredholm with Fredholm index $1$, both $1 - T_1(f) \cdot T_1(f)^\ast$ and $1 - T_1(f)^\ast \cdot T_1(f)$ are compact, and $T_1(f)$ commutes modulo compacts with multiplication by functions from $C[0, 2\pi]$. Furthermore, any two operators of the form $T_1(f)$ (for different $f$) are connected by a norm continuous homotopy consisting of operators having the same form. Finally, we define $T(f) := \begin{pmatrix}0 & T_1(f)^\ast \\ T_1(f) & 0\end{pmatrix} \in \IB(L^2[0,2\pi] \oplus L^2[0, 2\pi])$.

We assume the our homotopies $\rho_t$ and $T_t$ are constant in the intervals $[0, 2\pi/3]$ and $[4\pi / 3, 2\pi]$. Furthermore, we set
\[f(t) := \begin{cases} \cos 3t, & 0 \le t \le \pi / 3,\\ -1, & \pi/3 \le t \le 2 \pi.\end{cases}\]
Then $T_1(f)$ commutes with the projection $P$ onto $L^2[0, 2\pi / 3]$, $P \cdot T_1(f)$ is an operator of index $1$ on $L^2[0, 2\pi / 3]$, and $(1-P) T_1(f) \equiv -1$ on $L^2[2\pi/3, 2\pi]$. We choose $\alpha(t) \in C[0, 2\pi]$ with $0 \le \alpha(t) \le 1$, $\alpha(t) = 0$ for $t \le \pi / 3$, and $\alpha(t) = 1$ for $t \ge 2\pi / 3$. Using a norm continuous homotopy, we replace $N_1$ and $N_2$ by
\[\widetilde{N_1} := \sqrt{1 \hatotimes (1 - \alpha)} \cdot N_1 \cdot \sqrt{1 \hatotimes (1 - \alpha)}\]
and
\[\widetilde{N_2} := 1 \hatotimes \alpha + \sqrt{1 \hatotimes (1 - \alpha)} \cdot N_2 \cdot \sqrt{1 \hatotimes (1 - \alpha)}.\]
The operator $\widetilde{N_1}(T_t) + \widetilde{N_2}(1 \hatotimes T(f))$ commutes with $1 \hatotimes (P \oplus P)$ and we obtain for the decomposition $L^2[0,2\pi] \oplus L^2[0, 2\pi] = \image(P \oplus P) \oplus \image(1 - P \oplus P)$
\[\big( \IH, \rho_t, \widetilde{N_1}(T_t) + \widetilde{N_2}(1 \hatotimes T(f) \big) = \big( (H, \rho_0, T_0) \times [1] \big) \oplus \big(\text{degenerate}\big),\]
where $[1] \in K_0^u(\pt)$ is the multiplicative identity (see the third point of Theorem \ref{thm:external_prod_homology}) and recall that we assumed that $\rho_t$ and $T_t$ are constant in the intervals $[0, 2\pi/3]$ and $[4\pi / 3, 2\pi]$.

Setting
\[f(t) := \begin{cases} 1, & 0 \le t \le 5\pi / 3,\\ -\cos 3t, & 5\pi/3 \le t \le 2 \pi,\end{cases}\]
we get analogously
\[\big( \IH, \rho_t, \overline{N_1}(T_t) + \overline{N_2}(1 \hatotimes T(f) \big) = \big(\text{degenerate}\big) \oplus \big( (H, \rho_1, T_1) \times [1] \big),\]
for suitably defined operators $\overline{N_1}$ and $\overline{N_2}$ (their definition is similar to the one of $\widetilde{N_1}$ and $\widetilde{N_2}$). Putting all the homotopies of this proof together, we get that the modules $\big( (H, \rho_0, T_0) \times [1] \big) \oplus \big(\text{degenerate}\big)$ and $\big( (H, \rho_1, T_1) \times [1] \big) \oplus \big(\text{degenerate}\big)$ are operator homotopic, from which the claim follows.
\end{proof}

\subsection{Rough Baum--Connes conjecture}\label{sec:rough_BC}

\Spakula constructed in \cite[Section 9]{spakula_uniform_k_homology} the \emph{rough\footnote{We could have also called it the \emph{uniform coarse} assembly map, but the uniform coarse category is also called the rough category and therefore we stick to this shorter name.} assembly map}
\[\mu_u \colon K_\ast^u(X) \to K_\ast(C_u^\ast(Y)),\]
where $Y \subset X$ is a uniformly discrete quasi-lattice, $X$ a proper metric space, and $C_u^\ast(Y)$ the uniform Roe algebra of $Y$.\footnote{Recall that one possible model for the uniform Roe algebra $C_u^\ast(Y)$ is the norm closure of the $^\ast$-algebra of all finite propagation operators in $\IB(\ell^2(Y))$ with uniformly bounded coefficients. Another version is the norm closure of the $^\ast$-algebra of all finite propagation, uniformly locally compact operators in $\IB(\ell^2(Y)\otimes H)$ with uniformly bounded coefficients. \v{S}pakula--Willett \cite[Proposition 4.7]{spakula_willett_2} proved that these two versions are strongly Morita equivalent.} In this section we will discuss implications on the rough assembly map following from the properties of uniform $K$-homology that we have proved in the last sections.

Using homotopy invariance of uniform $K$-homology we will strengthen \v{S}pakula's results from \cite[Section 10]{spakula_uniform_k_homology}.

\begin{defn}[Rips complexes]
Let $Y$ be a discrete metric space and let $d \ge 0$. The \emph{Rips complex $P_d(Y)$ of $Y$} is a simplicial complex, where
\begin{itemize}
\item the vertex set of $P_d(Y)$ is $Y$, and
\item vertices $y_0, \ldots, y_q$ span a $q$-simplex if and only if we have $d(y_i, y_j) \le d$ for all $0 \le i, j \le q$.
\end{itemize}
Note that if $Y$ has coarsely bounded geometry, then the Rips complex $P_d(Y)$ is uniformly locally finite and finite dimensional and therefore also, especially, a simplicial complex of bounded geometry (i.e., the number of simplices in the link of each vertex is uniformly bounded). So if we equip $P_d(Y)$ with the metric derived from barycentric coordinates, $Y \subset P_d(Y)$ becomes a quasi-lattice (cf. Examples \ref{ex:coarsely_bounded_geometry}).
\end{defn}

Now we may state the \emph{rough Baum--Connes conjecture}:

\begin{conj}
Let $Y$ be a proper and uniformly discrete metric space with coarsely bounded geometry.

Then
\[\mu_u \colon \lim_{d \to \infty} K_\ast^u(P_d(Y)) \to K_\ast(C_u^\ast(Y))\]
is an isomorphism.
\end{conj}

Let us relate the conjecture quickly to manifolds of bounded geometry. First we need the following notion:

\begin{defn}[Equicontinuously contractible spaces]
A metric space $X$ is called \emph{equicontinuously contractible}, if for every radius $r > 0$ the collection of balls $\{B_r(x)\}_{x \in X}$ is equicontinuously contractible (i.e., the collection of the contracting homotopies is equicontinuous).\footnote{Equicontinuous contractibility is a slight strengthening of uniform contractibility: a metric space $X$ is called \emph{uniformly contractible}, if for every $r > 0$ there is an $s > 0$ such that every ball $B_r(x)$ can be contracted to a point in the ball $B_s(x)$.}
\end{defn}

\begin{example}
Universal covers of aspherical Riemannian manifolds equipped with the pull-back metric are equicontinuously contractible.
\end{example}

\begin{thm}
Let $M$ be an equicontinuously contractible manifold of bounded geometry and without boundary and let $Y \subset M$ be a uniformly discrete quasi-lattice in $M$.

Then we have a natural isomorphism
\[\lim_{d \to \infty} K^u_\ast(P_d(Y)) \cong K^u_\ast(M).\]
\end{thm}

The proof of this theorem is analogous to the corresponding non-uniform statement $\lim_{d \to \infty} K_\ast(P_d(Y)) \cong K_\ast(M)$ from \cite[Theorem 3.2]{yu_coarse_baum_connes_conj} and uses crucially the homotopy invariance of uniform $K$-homology.

Let us relate the rough Baum--Connes conjecture to the usual Baum--Connes conjecture: let $\Gamma$ be a countable, discrete group and denote by $|\Gamma|$ the metric space obtained by endowing $\Gamma$ with a proper, left-invariant metric. Then $|\Gamma|$ becomes a proper, uniformly discrete metric space with coarsely bounded geometry. Note that we can always find such a metric and that any two of such metrics are coarsely equivalent. If $\Gamma$ is finitely generated, an example is the word metric.

\Spakula proved in \cite[Corollary 10.3]{spakula_uniform_k_homology} the following equivalence of the rough Baum--Connes conjecture with the usual one: let $\Gamma$ be a torsion-free, countable, discrete group. Then the rough assembly map
\[\mu_u \colon \lim_{d \to \infty} K_\ast^u(P_d|\Gamma|) \to K_\ast(C_u^\ast|\Gamma|)\]
is an isomorphism if and only if the Baum--Connes assembly map
\[\mu \colon K_\ast^\Gamma(\underline{E}\Gamma; \ell^\infty(\Gamma)) \to K_\ast(C_r^\ast(\Gamma, \ell^\infty(\Gamma)))\]
for $\Gamma$ with coefficients in $\ell^\infty(\Gamma)$ is an isomorphism. For the definition of the Baum--Connes assembly map with coefficients the unfamiliar reader may consult the original paper \cite[Section 9]{baum_connes_higson}. Furthermore, the equivalence of the usual (i.e., non-uniform) coarse Baum--Connes conjecture with the Baum--Connes conjecture with coefficients in $\ell^\infty(\Gamma, \IK)$ was proved by Yu in \cite[Theorem 2.7]{yu_baum_connes_conj_coarse_geom}.

\Spakula mentioned in \cite[Remark 10.4]{spakula_uniform_k_homology} that the above equivalence does probably also hold without any assumptions on the torsion of $\Gamma$, but the proof of this would require some degree of homotopy invariance of uniform $K$-homology. So again we may utilize our proof of the homotopy invariance of uniform $K$-homology and therefore drop the assumption about the torsion of $\Gamma$.

\begin{thm}\label{thm:BC_equiv_uniform_coarse}
Let $\Gamma$ be a countable, discrete group.

Then the rough assembly map
\[\mu_u \colon \lim_{d \to \infty} K_\ast^u(P_d|\Gamma|) \to K_\ast(C_u^\ast|\Gamma|)\]
is an isomorphism if and only if the Baum--Connes assembly map
\[\mu \colon K_\ast^\Gamma(\underline{E}\Gamma; \ell^\infty(\Gamma)) \to K_\ast(C_r^\ast(\Gamma, \ell^\infty(\Gamma)))\]
for $\Gamma$ with coefficients in $\ell^\infty(\Gamma)$ is an isomorphism.
\end{thm}

\section{Uniform \texorpdfstring{$K$}{K}-theory}
\label{sec:uniform_k_th}

In this section we will define uniform $K$-theory and show that for \spinc manifolds it is \Poincare dual to uniform $K$-homology. The definition of uniform $K$-theory is based on the following observation: we want that it consists of vector bundles such that Dirac operators over manifolds of bounded geometry may be twisted with them (since we want a cap product between uniform $K$-homology and uniform $K$-theory). Hence we have to consider vector bundles of bounded geometry, because otherwise the twisted operator will not be uniform. See Definition~\ref{defn_bounded_geometry_manifolds} for the notion of bounded geometry.

The first guess is to use the algebra $C_b^\infty(M)$ of smooth functions on $M$ whose derivatives are all uniformly bounded, and then to consider its operator $K$-theory. This guess is based on the speculation that the boundedness of the derivatives translates into the boundedness of the Christoffel symbols if one equips the vector bundle with the induced metric and connection coming from the given embedding of the bundle into $\IC^k$ (this embedding is given to us because a projection matrix with entries in $C_b^\infty(M)$ defines a subbundle of $\IC^k$ by considering the image of the projection matrix). To our luck this first guess works out.

Note that other authors have, of course, investigated similar versions of $K$-theory: Kaad investigated in \cite{kaad} Hilbert bundles of bounded geometry over manifolds of bounded geometry (the author thanks Magnus Goffeng for pointing to that publication). Dropping the condition that the bundles must have bounded geometry, there is a general result by Morye contained in \cite{morye} having as a corollary the Serre--Swan theorem for smooth vector bundles over (possibly non-compact) smooth manifolds. If one is only interested in the last mentioned result, there is also the short note \cite{sardanashvily} by Sardanashvily.

\subsection{Definition and basic properties of uniform \texorpdfstring{$K$}{K}-theory}

As we have written above, we will define uniform $K$-theory of a manifold of bounded geometry as the operator $K$-theory of $C_b^\infty(M)$. But since $C_b^\infty(M)$ turns out to be a local $C^\ast$-algebra (see Lemma \ref{lem:C_b_infty_local}), its operator $K$-theory will coincide with the $K$-theory of its closure which is the $C^\ast$-algebra $C_u(M)$ of all bounded, uniformly continuous functions on $M$ (see Lemma \ref{lem:norm_completion_C_b_infty}). Hence we may define uniform $K$-theory for any metric space $X$ as the operator $K$-theory of $C_u(X)$.

\begin{defn}[Uniform $K$-theory]
Let $X$ be a metric space. The \emph{uniform $K$-theory groups of $X$} are defined as
\[K^p_u(X) := K_{-p}(C_u(X)),\]
where $C_u(X)$ is the $C^\ast$-algebra of bounded, uniformly continuous functions on $X$.
\end{defn}

The introduction of the minus sign in the index $-p$ in the above definition is just a convention which ensures that the indices in formulas, like the one for the cap product between uniform $K$-theory and uniform $K$-homology, coincide with the indices from the corresponding formulas for (co-)homology. Since complex $K$-theory is $2$-periodic, the minus sign does not change anything in the formulas.

Denoting by $\overline{X}$ the completion of the metric space $X$, we have $K^\ast_u(\overline{X}) = K^\ast_u(X)$ because every uniformly continuous function on $X$ has a unique extension to $\overline{X}$, i.e., $C_u(\overline{X}) = C_u(X)$. This means that, e.g., the uniform $K$-theories of the spaces $[0,1]$, $[0,1)$ and $(0,1)$ are all equal. Furthermore, since on a compact space $X$ we have $C_u(X) = C(X)$, uniform $K$-theory coincides for compact spaces with usual $K$-theory. Let us state this as a small lemma:

\begin{lem}
If $X$ is totally bounded, then $K^\ast_u(X) = K_u^\ast(\overline{X}) = K^\ast(\overline{X})$.\qed
\end{lem}

\begin{rem}
Note the following difference between uniform $K$-theory and uniform $K$-homology: whereas uniform $K$-theory of $X$ coincides with the uniform $K$-theory of the completion $\overline{X}$, this is in general not true for uniform $K$-homology.

Recall that in Proposition \ref{prop:compact_space_every_module_uniform} we have shown that if $X$ is totally bounded, then the uniform $K$-homology of $X$ coincides with locally finite $K$-homology of $X$. So for, e.g., an open ball $B$ in $\IR^n$ uniform and locally finite $K$-homology coincide and hence $K_m^u(B) \cong \IZ$ for $m=n$ and it vanishes for all other values of $m$. But due to homotopy invariance we have $K_m^u(\overline{B}) \cong K_m^u(*) \cong \IZ$ for $m=0$ and it vanishes for other values of $m$.

In the case of uniform $K$-theory we have $K^m_u(B) \cong K^m_u(\overline{B}) \cong K^m_u(*) \cong \IZ$ for $m=0$ and it vanishes otherwise.
\end{rem}

Recall that in Lemma \ref{lem:uniform_k_hom_discrete_space} we have shown that the uniform $K$-homology group $K_0^u(Y)$ of a uniformly discrete, proper metric space $Y$ of coarsely bounded geometry is isomorphic to the group $\ell^\infty_\IZ(Y)$ of all bounded, integer-valued sequences indexed by $Y$, and that $K_1^u(Y) = 0$. Since we want uniform $K$-theory to be \Poincare dual to uniform $K$-homology, we need the corresponding result for uniform $K$-theory.

\begin{lem}\label{lem:uniform_k_th_discrete_space}
Let $Y$ be a uniformly discrete metric space. Then $K^0_u(Y)$ is isomorphic to $\ell^\infty_\IZ(Y)$ and $K^1_u(Y) = 0$.
\end{lem}

\begin{proof}
The proof is an easy consequence of the fact that $C_u(Y) \cong \prod_{y \in Y} C(y) \cong \prod_{y \in Y} \IC$ for a uniformly discrete space $Y$, where the direct product of $C^\ast$-algebras is equipped with the sup-norm. The computation of the operator $K$-theory of $\prod_{y \in Y} \IC$ is now easily done (cf. \cite[Exercise 7.7.3]{higson_roe}).
\end{proof}

And last, we will give a relation of uniform $K$-theory with amenability. Analogous results for other uniform (co-)homology theories are known (see, e.g., \cite[Section 8]{block_weinberger_large_scale}).

\begin{lem}
Let $M$ be a metric space with amenable fundamental group.

We let $X$ be the universal cover of $M$ and we denote the covering projection by $\pi\colon X \to M$. Then the pull-back map $K^\ast_u(M) \to K^\ast_u(X)$ is injective.
\end{lem}

\begin{proof}
The projection $\pi$ induces a map $\pi^\ast \colon C_u(M) \to C_u(X)$ which then induces the pull-back map $K^\ast_u(M) \to K^\ast_u(X)$. We will prove the lemma by constructing a left inverse to the above map $\pi^\ast$, i.e., we will construct a map $p \colon C_u(X) \to C_u(M)$ with $p \circ \pi^\ast = \id \colon C_u(M) \to C_u(M)$.

Let $F \subset X$ be a fundamental domain for the action of the deck transformation group on $X$. Since $\pi_1(M)$ is amenable, we choose a \Folner sequence $(E_i)_i \subset \pi_1(M)$ in it. Now given a function $f \in C_u(X)$, we set
\[f_i(y) := \frac{1}{\card E_i} \sum_{x \in \pi^{-1}(y) \cap E_i \cdot F} f(x)\]
for $y \in M$. This gives us a sequence of functions $f_i$ on $M$, but they are in general not even continuous.

Now choosing a functional $\tau \in (\ell^\infty)^\ast$ associated to a free ultrafilter on $\IN$, we define $p(f)(y) := \tau(f_i(y))$. Due to the \Folner condition on $(E_i)_i$ all discontinuities that the functions $f_i$ may have vanish in the limit under $\tau$, and we get a bounded, uniformly continuous function $p(f)$ on $M$.

It is clear that $p$ is a left inverse to $\pi^\ast$.
\end{proof}

\subsection{Interpretation via vector bundles}\label{sec:interpretation_uniform_k_theory}

We will show now that if $M$ is a manifold of bounded geometry then we have a description of the uniform $K$-theory of $M$ via vector bundles of bounded geometry.

Let us first quickly recall the definition of bounded geometry for manifolds and vector bundles and discuss some examples.

\begin{defn}\label{defn_bounded_geometry_manifolds}
We will say that a Riemannian manifold $M$ has \emph{bounded geometry}, if
\begin{itemize}
\item the curvature tensor and all its derivatives are bounded, i.e., $\| \nabla^k \Rm (x) \| < C_k$ for all $x \in M$ and $k \in \IN_0$, and
\item the injectivity radius is uniformly positive, i.e., $\injrad_M(x) > \varepsilon > 0$ for all points $x \in M$ and for a fixed $\varepsilon > 0$.
\end{itemize}
If $E \to M$ is a vector bundle with a metric and compatible connection, we say that \emph{$E$ has bounded geometry}, if the curvature tensor of $E$ and all its derivatives are bounded.
\end{defn}

\begin{examples}
There are plenty of examples of manifolds of bounded geometry. The most important ones are coverings of compact Riemannian manifolds equipped with the pull-back metric, homogeneous manifolds with an invariant metric, and leafs in a foliation of a compact Riemannian manifold (this is proved by Greene in \cite[lemma on page 91 and the paragraph thereafter]{greene}).

For vector bundles, the most important examples are of course again pull-back bundles of bundles over compact manifolds equipped with the pull-back metric and connection, and the tangent bundle of a manifold of bounded geometry.

Furthermore, if $E$ and $F$ are two vector bundles of bounded geometry, then the dual bundle $E^\ast$, the direct sum $E \oplus F$, the tensor product $E \otimes F$ (and so especially also the homomorphism bundle $\Hom(E, F) = F \otimes E^\ast$) and all exterior powers $\Lambda^l E$ are also of bounded geometry. If $E$ is defined over $M$ and $F$ over $N$, then their external tensor product\footnote{The fiber of $E \boxtimes F$ over the point $(x,y) \in M \times N$ is given by $E_x \otimes F_y$.} $E \boxtimes F$ over $M \times N$ is also of bounded geometry.
\end{examples}

Greene proved in \cite[Theorem 2']{greene} that there are no obstructions against admitting a metric of bounded geometry, i.e., every smooth manifold without boundary admits one. On manifolds of bounded geometry there is also no obstruction for a vector bundle to admit a metric and compatible connection of bounded geometry. The construction of the metric and the connection is done in a uniform covering of $M$ by normal coordinate charts and subordinate uniform partition of unity (we will discuss these things in a moment) and we have to use the local characterization of bounded geometry for vector bundles from Lemma \ref{lem:equiv_characterizations_bounded_geom_bundles}.

The first step in showing that uniform $K$-theory has an interpretation via vector bundles of bounded geometry is to show that the operator $K$-theory of $C_u(M)$ coincides with the operator $K$-theory of $C_b^\infty(M)$. This is established via the following two lemmas.

\begin{lem}\label{lem:C_b_infty_local}
Let $M$ be a manifold of bounded geometry.

Then $C_b^\infty(M)$ is a local $C^\ast$-algebra\footnote{That is to say, it and all matrix algebras over it are closed under holomorphic functional calculus and its completion is a $C^\ast$-algebra.}.
\end{lem}

\begin{proof}
Since $C_b^\infty(M)$ is a $^\ast$-subalgebra of the $C^\ast$-algebra $C_b(M)$ of bounded continuous functions on $M$, then norm completion of $C_b^\infty(M)$, i.e., its closure in $C_b(M)$, is surely a $C^\ast$-algebra.

So we have to show that $C_b^\infty(M)$ and all matrix algebras over it are closed under holomorphic functional calculus. Since $C_b^\infty(M)$ is naturally a \Frechet algebra with a \Frechet topology which is finer than the sup-norm topology, by \cite[Corollary 2.3]{schweitzer}\footnote{The corollary states that under the condition that the topology of a \Frechet algebra $A$ is finer than the sup-norm topology we may conclude that if $A$ is closed under holomorphic functional calculus, then this holds also for all matrix algebras over $A$.} it remains to show that $C_b^\infty(M)$ itself is closed under holomorphic functional calculus.

But that $C_b^\infty(M)$ is closed under holomorphic functional calculus is easily seen using \cite[Lemma 1.2]{schweitzer}, which states that a unital \Frechet algebra $A$ with a topology finer than the sup-norm topology is closed under functional calculus if and only if the inverse $a^{-1} \in \overline{A}$ of any invertible element $a \in A$ actually lies in $A$.
\end{proof}

For the proof of Lemma~\ref{lem:norm_completion_C_b_infty} we need the next Lemma~\ref{lem:nice_coverings_partitions_of_unity} about manifolds of bounded geometry. A proof of it may be found in, e.g., \cite[Appendix A1.1]{shubin} (Shubin addresses the first statement about the existence of the covers to the paper \cite{gromov_curvature_diameter_betti_numbers} of Gromov).

\begin{lem}\label{lem:nice_coverings_partitions_of_unity}
Let $M$ be a manifold of bounded geometry.

For every $0 < \varepsilon < \tfrac{\injrad_M}{3}$ there exists a covering of $M$ by normal coordinate charts of radius $\varepsilon$ with the properties that the midpoints of the charts form a uniformly discrete set in $M$ and that the coordinate charts with double radius $2\varepsilon$ form a uniformly locally finite cover of $M$.

Furthermore, there is a subordinate partition of unity $1 = \sum_i \varphi_i$ with $\supp \varphi_i \subset B_{2\varepsilon}(x_i)$, such that in normal coordinates the functions $\varphi_i$ and all their derivatives are uniformly bounded (i.e., the bounds do not depend on $i$).\qed
\end{lem}

\begin{lem}\label{lem:norm_completion_C_b_infty}
Let $M$ be a manifold of bounded geometry.

Then the sup-norm completion of $C_b^\infty(M)$ is the $C^\ast$-algebra $C_u(M)$ of bounded, uniformly continuous functions on $M$.
\end{lem}

\begin{proof}
We surely have $\overline{C_b^\infty(M)} \subset C_u(M)$. To show the converse inclusion, we have to approximate a bounded, uniformly continuous function by a smooth one with bounded derivatives. This can be done by choosing a nice cover of $M$ with corresponding nice subordinate partitions of unity via Lemma \ref{lem:nice_coverings_partitions_of_unity} and then apply in every coordinate chart the same mollifier to the uniformly continous function.

Let us elaborate a bit more on the last sentence of the above paragraph: after choosing the nice cover and cutting a function $f \in C_u(M)$ with the subordinate partition of unity $\{\varphi_i\}$, we have transported the problem to Euclidean space $\IR^n$ and our family of functions $\varphi_i f$ is uniformly equicontinuous (this is due to the uniform continuity of $f$ and will be crucially important at the end of this proof). Now let $\psi$ be a mollifier on $\IR^n$, i.e., a smooth function with $\psi \ge 0$, $\supp \psi \subset B_1(0)$, $\int_{\IR^n} \psi d\lambda = 1$ and $\psi_\varepsilon := \varepsilon^{-n} \psi(\largecdot / \varepsilon) \stackrel{\varepsilon \to 0}\longrightarrow \delta_0$. Since convolution satisfies $D^\alpha (\varphi_i f \ast \psi_\varepsilon) = \varphi_i f \ast D^\alpha \psi_\varepsilon$, where $D^\alpha$ is a directional derivative on $\IR^n$ in the directions of the multi-index $\alpha$ and of order $|\alpha|$, we conclude that every mollified function $\varphi_i f \ast \psi_\varepsilon$ is smooth with bounded derivatives. Furthermore, we know $\| \varphi_i f \ast D^\alpha \psi_\epsilon \|_\infty \le \| \varphi_i f \|_\infty \cdot \| D^\alpha \psi_\varepsilon \|_1$ from which we conclude that the bounds on the derivatives of $\varphi_i f \ast \psi_\varepsilon$ are uniform in $i$, i.e., if we glue the functions $\varphi_i f \ast \psi_\epsilon$ together to a function on the manifold $M$ (note that the functions $\varphi_i f \ast \psi_\epsilon$ are supported in our chosen nice cover since convolution with $\psi_\varepsilon$ enlarges the support at most by $\varepsilon$), we get a function $f_\varepsilon \in C_b^\infty(M)$. It remains to show that $f_\varepsilon$ converges to $f$ in sup-norm, which is equivalent to the statement that $\varphi_i f \ast \psi_\varepsilon$ converges to $\varphi_i f$ in sup-norm and uniformly in the index $i$. But we know that
\[ \big| (\varphi_i f \ast \psi_\epsilon) (x) - (\varphi_i f) (x) \big| \le \sup_{\substack{x \in \supp \varphi_i f\\y \in B_\varepsilon(0)}} \big| (\varphi_i f) (x - y) - (\varphi_i f) (x) \big| \]
from which the claim follows since the family of functions $\varphi_i f$ is uniformly equicontinuous (recall that this followed from the uniform continuity of $f$ and this here is actually the only point in this proof where we need that property of $f$).
\end{proof}

Since $C_b^\infty(M)$ is an $m$-convex \Frechet algebra\footnote{That is to say, a \Frechet algebra such that its topology is given by a countable family of submultiplicative seminorms.}, we can also use the $K$-theory for $m$-convex \Frechet algebras as developed by Phillips in \cite{phillips} to define the $K$-theory groups of $C_b^\infty(M)$. But this produces the same groups as the operator $K$-theory, since $C_b^\infty(M)$ is an $m$-convex \Frechet algebra with a finer topology than the norm topology and therefore its $K$-theory for $m$-convex \Frechet algebras coincides with its operator $K$-theory by \cite[Corollary 7.9]{phillips}.

We summarize this observations in the following lemma:

\begin{lem}\label{lem:equivalent_defns_uniform_k_theory}
Let $M$ be a manifold of bounded geometry.

Then the operator $K$-theory of $C_u(M)$, the operator $K$-theory of $C_b^\infty(M)$ and Phillips $K$-theory for $m$-convex \Frechet algebras of $C_b^\infty(M)$ are all pairwise isomorphic.\qed
\end{lem}

So we have shown $K^\ast_u(M) \cong K_{-\ast}(C_b^\infty(M))$. In order to conclude the description via vector bundles of bounded geometry, we will need to establish the correspondence between vector bundles of bounded geometry and idempotent matrices with entries in $C_b^\infty(M)$. This will be done in the next subsections.

\subsubsection*{Isomorphism classes and complements}

Let $M$ be a manifold of bounded geometry and $E$ and $F$ two complex vector bundles equipped with Hermitian metrics and compatible connections.

\begin{defn}[$C^\infty$-boundedness / $C_b^\infty$-isomorphy of vector bundle homomorphisms]\label{defn:C_infty_bounded}
We will call a vector bundle homomorphism $\varphi\colon E \to F$ \emph{$C^\infty$-bounded}, if with respect to synchronous framings of $E$ and $F$ the matrix entries of $\varphi$ are bounded, as are all their derivatives, and these bounds do not depend on the chosen base points for the framings or the synchronous framings themself.

$E$ and $F$ will be called \emph{$C_b^\infty$-isomorphic}, if there exists an isomorphism $\varphi\colon E \to F$ such that both $\varphi$ and $\varphi^{-1}$ are $C^\infty$-bounded. In that case we will call the map $\varphi$ a $C_b^\infty$-isomorphism. Often we will write $E \cong F$ when no confusion can arise with mistaking it with algebraic isomorphy.
\end{defn}

Using the characterization of bounded geometry via the matrix transition functions from the next Lemma \ref{lem:equiv_characterizations_bounded_geom_bundles}, we immediately see that if $E$ and $F$ are $C_b^\infty$-isomorphic, than $E$ is of bounded geometry if and only if $F$ is. The equivalence of the first two bullet points in the next lemma is stated in, e.g., \cite[Proposition 2.5]{roe_index_1}. Concerning the third bullet point, the author could not find any citable reference in the literature (though Shubin uses in \cite{shubin} this as the actual definition).

\begin{lem}\label{lem:equiv_characterizations_bounded_geom_bundles}
Let $M$ be a manifold of bounded geometry and $E \to M$ a vector bundle. Then the following are equivalent:

\begin{itemize}
\item $E$ has bounded geometry,
\item the Christoffel symbols $\Gamma_{i \alpha}^\beta(y)$ of $E$ with respect to synchronous framings (considered as functions on the domain $B$ of normal coordinates at all points) are bounded, as are all their derivatives, and this bounds are independent of $x \in M$, $y \in \exp_x(B)$ and $i, \alpha, \beta$, and
\item the matrix transition functions between overlapping synchronous framings are uniformly bounded, as are all their derivatives (i.e., the bounds are the same for all transition functions).\qed
\end{itemize}
\end{lem}

It is clear that $C_b^\infty$-isomorphy is compatible with direct sums and tensor products, i.e., if $E \cong E^\prime$ and $F \cong F^\prime$ then $E \oplus F \cong E^\prime \oplus F^\prime$ and $E \otimes F \cong E^\prime \otimes F^\prime$.

We will now give a useful global characterization of $C_b^\infty$-isomorphisms if the vector bundles have bounded geometry:

\begin{lem}\label{lem:C_b_infty_Iso_equivalent}
Let $E$ and $F$ have bounded geometry and let $\varphi\colon E \to F$ be an isomorphism. Then $\varphi$ is a $C_b^\infty$-isomorphism if and only if
\begin{itemize}
\item $\varphi$ and $\varphi^{-1}$ are bounded, i.e., $\|\varphi(v)\| \le C \cdot \|v\|$ for all $v \in E$ and a fixed $C > 0$ and analogously for $\varphi^{-1}$, and
\item $\nabla^E - \varphi^\ast \nabla^F$ is bounded and also all its covariant derivatives.
\end{itemize}
\end{lem}

\begin{proof}
For a point $p \in M$ let $B \subset M$ be a geodesic ball centered at $p$, $\{ x_i \}$ the corresponding  normal coordinates of $B$, and let $\{ E_\alpha(y) \}$, $y \in B$, be a framing for $E$. Then we may write every vector field $X$ on $B$ as $X = X^i \frac{\partial}{\partial x_i} = (X^1 , \ldots, X^n)^T$ and every section $e$ of $E$ as $e = e^\alpha E_\alpha = (e^1, \ldots, e^k)^T$, where we assume the Einstein summation convention and where $\largecdot^T$ stands for the transpose of the vector (i.e., the vectors are actually column vectors). Furthermore, after also choosing a framing for $F$, $\varphi$ becomes a matrix for every $y \in B$ and $\varphi(e)$ is then just the matrix multiplication $\varphi(e) = \varphi \cdot e$. Finally, $\nabla^E_X e$ is locally given by
\[\nabla^E_X e = X(e) + \Gamma^E(X)\cdot e,\]
where $X(e)$ is the column vector that we get after taking the derivative of every entry $e^j$ of $e$ in the direction of $X$ and $\Gamma^E$ is a matrix of $1$-forms (i.e., $\Gamma^E(X)$ is then a usual matrix that we multiply with the vector $e$). The entries of $\Gamma^E$ are called the connection $1$-forms.

Since $\varphi$ is an isomorphism, the pull-back connection $\varphi^\ast \nabla^F$ is given by\footnote{Note that $\varphi$ is a morphism of vector bundles, i.e., the following diagram commutes:
\[\xymatrix{E \ar[rr]^\varphi \ar[dr] & & F \ar[dl]\\ & M &}\]
This means that $\varphi$ descends to the identity on $M$, i.e., in Equation \eqref{eq:vect_iso} the vector field $X$ occurs on both the left and the right hand side (since actually we have $(\varphi^{-1})^\ast X$ on the right hand side).}
\begin{equation}
\label{eq:vect_iso}
(\varphi^\ast \nabla^F)_X e = \varphi^\ast (\nabla^F_X (\varphi^{-1})^\ast e),
\end{equation}
so that locally we get
\[(\varphi^\ast \nabla^F)_X e = \varphi^{-1}\cdot \big( X(\varphi \cdot e) + \Gamma^F(X) \cdot \varphi \cdot e\big).\]
Using the product rule we may rewrite $X(\varphi \cdot e) = X(\varphi) \cdot e + \varphi \cdot X(e)$, where $X(\varphi)$ is the application of $X$ to every entry of $\varphi$. So at the end we get for the difference $\nabla^E - \varphi^\ast \nabla^F$ in local coordinates and with respect to framings of $E$ and $F$
\begin{equation}\label{eq:difference_connections_local}
(\nabla^E - \varphi^\ast \nabla^F)_X e = \Gamma^E(X) \cdot e - \varphi^{-1} \cdot X(\varphi) \cdot e - \varphi^{-1} \cdot \Gamma^F(X) \cdot \varphi \cdot e.
\end{equation}

Since $E$ and $F$ have bounded geometry, by Lemma \ref{lem:equiv_characterizations_bounded_geom_bundles} the Christoffel symbols of them with respect to synchronous framings are bounded and also all their derivatives, and these bounds are independent of the point $p \in M$ around that we choose the normal coordinates and the framings. Assuming that $\varphi$ is a $C_b^\infty$-isomorphism, the same holds for the matrix entries of $\varphi$ and $\varphi^{-1}$ and we conclude with the above Equation \eqref{eq:difference_connections_local} that the difference $\nabla^E - \varphi^\ast \nabla^F$ is bounded and also all its covariant derivatives (here we also need to consult the local formula for covariant derivatives of tensor fields).

Conversely, assume that $\varphi$ and $\varphi^{-1}$ are bounded and that the difference $\nabla^E - \varphi^\ast \nabla^F$ is bounded and also all its covariant derivatives. If we denote by $\Gamma^{\text{diff}}$ the matrix of $1$-forms given by
\[\Gamma^{\text{diff}}(X) = \Gamma^E(X) - \varphi^{-1} \cdot X(\varphi) - \varphi^{-1} \cdot \Gamma^F(X) \cdot \varphi,\]
we get from Equation \eqref{eq:difference_connections_local}
\[X(\varphi) = \varphi \cdot (\Gamma^E(X) - \Gamma^{\text{diff}}(X)) - \Gamma^F(X) \cdot \varphi.\]
Since we assumed that $\varphi$ is bounded, its matrix entries must be bounded. From the above equation we then conclude that also the first derivatives of these matrix entries are bounded. But now that we know that the entries and also their first derivatives are bounded, we can differentiate the above equation once more to conclude that also the second derivatives of the matrix entries of $\varphi$ are bounded, on so on. This shows that $\varphi$ is $C^\infty$-bounded. At last, it remains to see that the matrix entries of $\varphi^{-1}$ and also all their derivatives are bounded. But since locally $\varphi^{-1}$ is the inverse matrix of $\varphi$, we just have to use Cramer's rule.
\end{proof}

An important property of vector bundles over compact spaces is that they are always complemented, i.e., for every bundle $E$ there is a bundle $F$ such that $E \oplus F$ is isomorphic to the trivial bundle. Note that this fails in general for non-compact spaces. So our important task is now to show that we have an analogous proposition for vector bundles of bounded geometry, i.e., that they are always complemented (in a suitable way).

\begin{defn}[$C_b^\infty$-complemented vector bundles]
A vector bundle $E$ will be called \emph{$C_b^\infty$-complemented}, if there is some vector bundle $E^\perp$ such that $E \oplus E^\perp$ is $C_b^\infty$-isomorphic to a trivial bundle with the flat connection.
\end{defn}

Since a bundle with a flat connection is trivially of bounded geometry, we get that $E \oplus E^\perp$ is of bounded geometry. And since a direct sum $E \oplus E^\perp$ of vector bundles is of bounded geometry if and only if both vector bundles $E$ and $E^\perp$ are of bounded geometry, we conclude that if $E$ is $C_b^\infty$-complemented, then both $E$ and its complement $E^\perp$ are of bounded geometry. It is also clear that if $E$ is $C_b^\infty$-complemented and $F \cong E$, then $F$ is also $C_b^\infty$-complemented.

We will now prove the crucial fact that every vector bundle of bounded geometry is $C_b^\infty$-complemented. The proof is just the usual one for vector bundles over compact Hausdorff spaces, but we additionally have to take care of the needed uniform estimates. As a source for this usual proof the author used \cite[Proposition 1.4]{hatcher_VB}. But first we will need a technical lemma.

\begin{lem}\label{lem:coloring_graph}
Let a covering $\{U_\alpha\}$ of $M$ with finite multiplicity be given. Then there exists a coloring of the subsets $U_\alpha$ with finitely many colors such that no two intersecting subsets have the same color.
\end{lem}

\begin{proof}
Construct a graph whose vertices are the subsets $U_\alpha$ and two vertices are connected by an edge if the corresponding subsets intersect. We have to find a coloring of this graph with only finitely many colors where connected vertices do have different colors.

To do this, we firstly use the theorem of de Bruijin--Erd\"{o}s stating that an infinite graph may be colored by $k$ colors if and only if every of its finite subgraphs may be colored by $k$ colors (one can use the Lemma of Zorn to prove this).

Secondly, since the covering has finite multiplicity it follows that the number of edges attached to each vertex in our graph is uniformly bounded from above, i.e., the maximum vertex degree of our graph is finite. But this also holds for every subgraph of our graph, with the maximum vertex degree possibly only decreasing by passing to a subgraph. Now a simple greedy algorithm shows that every finite graph may be colored with one more color than its maximum vertex degree: just start by coloring a vertex with some color, go to the next vertex and use an admissible color for it, and so on.
\end{proof}

\begin{prop}\label{prop:every_bundle_complemented}
Let $M$ be a manifold of bounded geometry and let $E \to M$ be a vector bundle of bounded geometry.

Then $E$ is $C_b^\infty$-complemented.
\end{prop}

\begin{proof}
Since $M$ and $E$ have bounded geometry, we can find a uniformly locally finite cover of $M$ by normal coordinate balls of a fixed radius together with a subordinate partition of unity whose derivatives are all uniformly bounded and such that over each coordinate ball $E$ is trivialized via a synchronous framing. This follows basically from Lemma \ref{lem:nice_coverings_partitions_of_unity}.

Now we the above Lemma \ref{lem:coloring_graph} to color the coordinate balls with finitely many colors so that no two balls with the same color do intersect. This gives a partition of the coordinate balls into $N$ families $U_1, \ldots, U_N$ such that every $U_i$ is a collection of disjoint balls, and we get a corresponding subordinate partition of unity $1 = \varphi_1 + \ldots + \varphi_N$ with uniformly bounded derivatives (each $\varphi_i$ is the sum of all the partition of unity functions of the coordinate balls of $U_i$). Furthermore, $E$ is trivial over each $U_i$ and we denote these trivializations coming from the synchronous framings by $h_i \colon p^{-1}(U_i) \to U_i \times \IC^k$, where $p\colon E \to M$ is the projection.

Now we set
\[g_i\colon E \to \IC^k, \ g_i(v) := \varphi_i(p(v)) \cdot \pi_i (h_i (v)),\]
where $\pi_i \colon U_i \times \IC^k \to \IC^k$ is the projection. Each $g_i$ is a linear injection on each fiber over $\varphi_i^{-1}(0,1]$ and so, if we define
\[g \colon E \to \IC^{Nk}, \ g(v) := (g_1(v), \ldots, g_N(v)),\]
we get a map $g$ that is a linear injection on each fiber of $E$. Finally, we define a map
\[G\colon E \to M \times \IC^{Nk}, \ G(v) := (p(v), g(v)).\]
This establishes $E$ as a subbundle of a trivial bundle.

If we equip $M \times \IC^{Nk}$ with a constant metric and the flat connection, we get that the induced metric and connection on $E$ is $C_b^\infty$-isomorphic to the original metric and connection on $E$ (this is due to our choice of $G$). Now let us denote by $e$ the projection matrix of the trivial bundle $\IC^{Nk}$ onto the subbundle $G(E)$ of it, i.e., $e$ is an $Nk \times Nk$-matrix with functions on $M$ as entries and $\image e = E$. Now, again due to our choice of $G$, we can conclude that these entries of $e$ are bounded functions with all derivatives of them also bounded, i.e., $e \in \Idem_{Nk \times Nk}(C_b^\infty(M))$. Now the claim follows with the Proposition \ref{prop:image_proj_matrix_complemented} which establishes the orthogonal complement $E^\perp$ of $E$ in $\IC^{Nk}$ with the induced metric and connection as a $C_b^\infty$-complement to $E$.
\end{proof}

We have seen in the above proposition that every vector bundle of bounded geometry is $C_b^\infty$-complemented. Now if we have a manifold of bounded geometry $M$, then its tangent bundle $TM$ is of bounded geometry and so we know that it is $C_b^\infty$-complemented (although $TM$ is real and not a complex bundle, the above proof of course also holds for real vector bundles). But in this case we usually want the complement bundle to be given by the normal bundle $NM$ coming from an embedding $M \hookrightarrow \IR^N$. We will prove this now under the assumption that the embedding of $M$ into $\IR^N$ is ``nice'':\footnote{See \cite{MO_iso_embedding_bounded_second_form} for a discussion of existence of ``nice'' embeddings.}

\begin{cor}\label{cor:tangent_bundle_complemented}
Let $M$ be a manifold of bounded geometry and let it be isometrically embedded into $\IR^N$ such that the second fundamental form is $C^\infty$-bounded.

Then its tangent bundle $TM$ is $C_b^\infty$-complemented by the normal bundle $NM$ corresponding to this embedding $M \hookrightarrow \IR^N$, equipped with the induced metric and connection.
\end{cor}

\begin{proof}
Let $M$ be isometrically embedded in $\IR^N$. Then its tangent bundle $TM$ is a subbundle of $T\IR^N$ and we denote the projection onto it by $\pi\colon T\IR^N \to TM$. Because of Point 1 of the following Proposition \ref{prop:image_proj_matrix_complemented} it suffices to show that the entries of $\pi$ are $C^\infty$-bounded functions.

Let $\{v_i\}$ be the standard basis of $\IR^N$ and let $\{E_\alpha(y)\}$ be the orthonormal frame of $TM$ arising out of normal coordinates $\{\partial_k\}$ of $M$ via the Gram-Schmidt process. Then the entries of the projection matrix $\pi$ with respect to the basis $\{v_i\}$ are given by
\[\pi_{ij}(y) = \sum_\alpha \langle E_\alpha(y), v_j\rangle \langle E_\alpha(y), v_i\rangle.\]

Let $\widetilde{\nabla}$ denote the flat connection on $\IR^N$. Since $\widetilde{\nabla}_{\partial_k} v_i = 0$ we get
\[\partial_k \pi_{ij} (y) = \sum_\alpha \langle \widetilde{\nabla}_{\partial_k} E_\alpha (y), v_j\rangle \langle E_\alpha(y), v_i\rangle + \langle E_\alpha(y), v_j\rangle \langle \widetilde{\nabla}_{\partial_k} E_\alpha (y), v_i\rangle.\]
Now if we denote by $\nabla^M$ the connection on $M$, we get
\[\widetilde{\nabla}_{\partial_k} E_\alpha(y) = \nabla^{M}_{\partial_k} E_\alpha(y) + \operatorname{I\!\!\;I}(\partial_k, E_\alpha),\]
where $\operatorname{I\!\!\;I}$ is the second fundamental form. So to show that $\pi_{ij}$ is $C^\infty$-bounded, we must show that $E_\alpha(y)$ are $C^\infty$-bounded sections of $TM$ (since by assumption the second fundamental form is a $C^\infty$-bounded tensor field). But that these $E_\alpha(y)$ are $C^\infty$-bounded sections of $TM$ follows from their construction (i.e., applying Gram-Schmidt to the normal coordinate fields $\partial_k$) and because $M$ has bounded geometry.
\end{proof}

\subsubsection*{Interpretation of \texorpdfstring{$K^0_u(M)$}{even uniform K(M)}}

Recall for the understanding of the following proposition the fact that if a vector bundle is $C_b^\infty$-complemented, then it is of bounded geometry. Furthermore, this proposition is the crucial one that gives us the description of uniform $K$-theory via vector bundles of bounded geometry.

\begin{prop}\label{prop:image_proj_matrix_complemented}
Let $M$ be a manifold of bounded geometry.
\begin{enumerate}
\item Let $e \in \Idem_{N \times N}(C_b^\infty(M))$ be an idempotent matrix.

Then the vector bundle $E := \image e$, equipped with the induced metric and connection, is $C_b^\infty$-complemented.

\item Let $E$ be a $C_b^\infty$-complemented vector bundle, i.e., there is a vector bundle $E^\perp$ such that $E \oplus E^\perp$ is $C_b^\infty$-isomorphic to the trivial $N$-dimensional bundle $\IC^N \to M$.

Then all entries of the projection matrix $e$ onto the subspace $E \oplus 0 \subset \IC^N$ with respect to a global synchronous framing of $\IC^N$ are $C^\infty$-bounded, i.e., we have $e \in \Idem_{N \times N}(C_b^\infty(M))$.
\end{enumerate}
\end{prop}

\begin{proof}[Proof of point 1]
We denote by $E$ the vector bundle $E := \image e$ and by $E^\perp$ its complement $E^\perp := \image (1-e)$ and equip them with the induced metric and connection. So we have to show that $E \oplus E^\perp$ is $C_b^\infty$-isomorphic to the trivial bundle $\IC^N \to M$.

Let $\varphi\colon E \oplus E^\perp \to \IC^N$ be the canonical algebraic isomorphism $\varphi(v,w) := v + w$. We have to show that both $\varphi$ and $\varphi^{-1}$ are $C^\infty$-bounded. 

Let $p \in M$. Let $\{ E_\alpha \}$ be an orthonormal basis of the vector space $E_p$ and $\{ E^\perp_\beta \}$ an orthonormal basis of $E^\perp_p$. Then the set $\{ E_\alpha, E^\perp_\beta \}$ is an orthonormal basis for $\IC_p^N$. We extend $\{E_\alpha\}$ to a synchronous framing $\{E_\alpha(y)\}$ of $E$ and $\{E^\perp_\beta\}$ to a synchronous framing $\{E^\perp_\beta(y)\}$ of $E^\perp$. Since $\IC^N$ is equipped with the flat connection, the set $\{ E_\alpha, E^\perp_\beta \}$ forms a synchronous framing for $\IC^N$ at all points of the normal coordinate chart. Then $\varphi(y)$ is the change-of-basis matrix from the basis $\{E_\alpha(y), E_\beta^\perp(y)\}$ to the basis $\{ E_\alpha, E^\perp_\beta \}$ and vice versa for $\varphi^{-1}(y)$; see Figure \ref{fig:frames}:

\begin{figure}[htbp]
\centering
\includegraphics[scale=0.7]{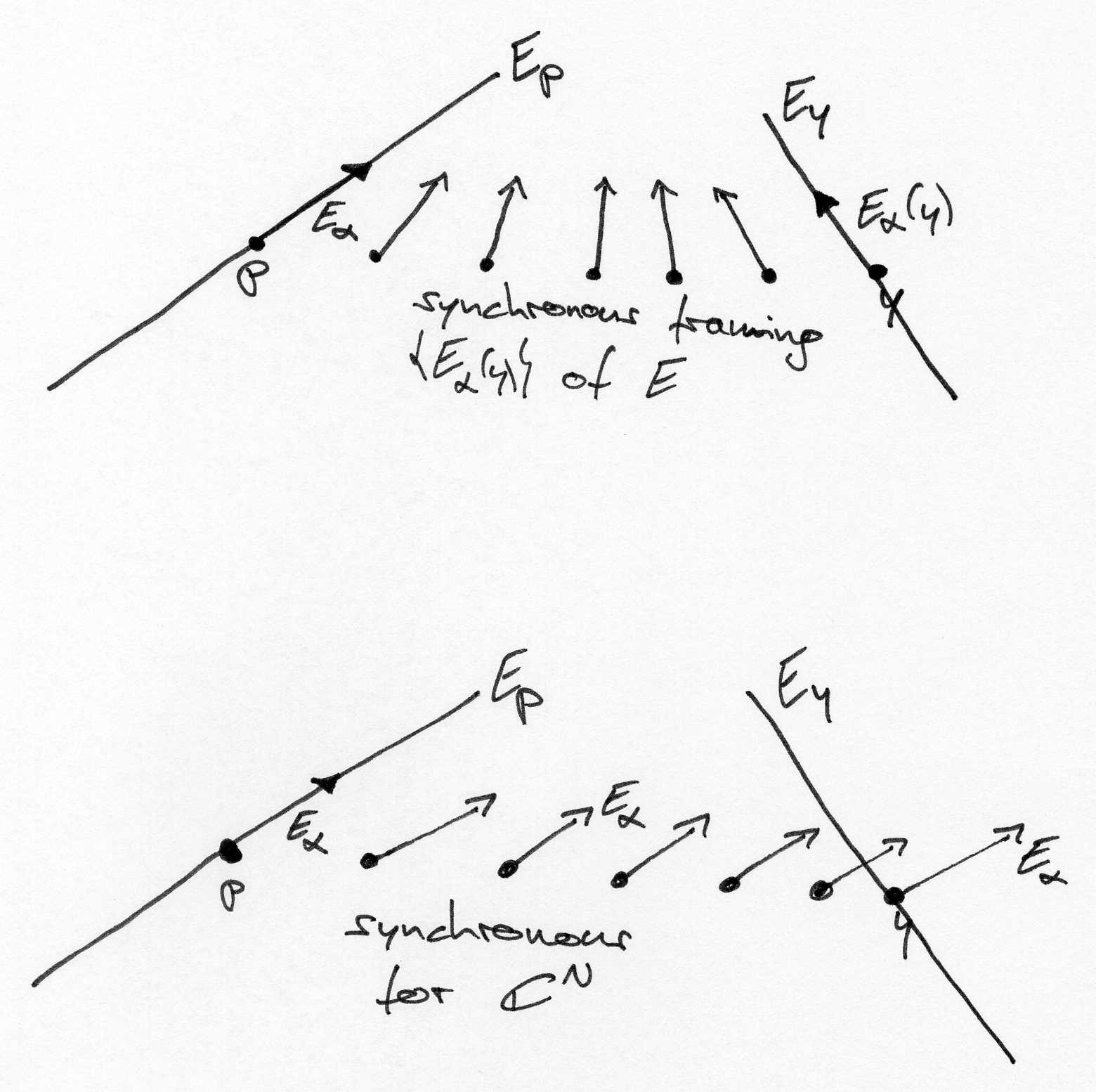}
\caption{The frames $\{E_\alpha(y)\}$ and $\{E_\alpha\}$.}
\label{fig:frames}
\end{figure}

We have $e(p)(E_\alpha) = E_\alpha$. Since the entries of $e$ are $C^\infty$-bounded and the rank of a matrix is a lower semi-continuous function of the entries, there is some geodesic ball $B$ around $p$ such that $\{ e(y)(E_\alpha) \}$ forms a basis of $E_y$ for all $y \in B$ and the diameter of the ball $B$ is bounded from below independently of $p \in M$. We denote by $\Gamma_{i \nu}^\mu(y)$ the Christoffel symbols of $E$ with respect to the frame $\{ e(y)(E_\alpha) \}$. Let $\gamma(t)$ be a radial geodesic in $M$ with $\gamma(0) = p$. If we now let $E_\alpha(\gamma(t))^\mu$ denote the $\mu$th entry of the vector $E_\alpha(\gamma(t))$ represented in the basis $\{ e(\gamma(t))(E_\alpha) \}$, then (since it is a synchronous frame) it satisfies the ODE
\[ \tfrac{d}{dt} E_\alpha(\gamma(t))^\mu = -\sum_{i,\nu} E_\alpha(\gamma(t))^\nu \cdot \tfrac{d}{dt}{\gamma_i}(t) \cdot \Gamma_{i \nu}^\mu(\gamma(t)),\]
where $\{\gamma_i\}$ is the coordinate representation of $\gamma$ in normal coordinates $\{x_i\}$. Since $\gamma$ is a radial geodesic, its representation in normal coordinates is $\gamma_i(t) = t \cdot \gamma_i(0)$ and so the above formula simplifies to
\begin{equation}\label{eq:ODE_synchronous_frame}
\tfrac{d}{dt} E_\alpha(\gamma(t))^\mu = -\sum_{i,\nu} E_\alpha(\gamma(t))^\nu \cdot \gamma_i(0) \cdot \Gamma_{i \nu}^\mu(\gamma(t)).
\end{equation}

Since $\Gamma_{i \nu}^\mu(y)$ are the Christoffel symbols with respect to the frame $\{ e(y)(E_\alpha) \}$, we get the equation
\begin{equation}\label{eq:christoffel_symbols_representation}
\sum_{\mu} \Gamma_{i \nu}^\mu(y) \cdot e(y)(E_\mu) = \nabla^E_{\partial_i} e(y)(E_\nu).
\end{equation}
Now using that $\nabla^E$ is induced by the flat connection, we get
\[\nabla^E_{\partial_i} e(y)(E_\nu) = e (\partial_i(e(y)(E_\nu))) = e((\partial_i e)(y)(E_\nu)),\]
i.e., $e((\partial_i e)(y)(E_\nu))$ is the representation of $\nabla^E_{\partial_i} e(y)(E_\nu)$ with respect to the frame $\{E_\alpha, E_\beta^\perp\}$. Since the entries of $e$ are $C^\infty$-bounded, the entries of this representation $e((\partial_i e)(y)(E_\nu))$ are also $C^\infty$-bounded. From Equation \eqref{eq:christoffel_symbols_representation} we see that $\Gamma_{i \nu}^\mu(y)$ is the representation of $\nabla^E_{\partial_i} e(y)(E_\nu)$ in the frame $\{e(y)(E_\mu)\}$. So we conclude that the Christoffel symbols $\Gamma_{i \nu}^\mu(y)$ are $C^\infty$-bounded functions.

Equation \eqref{eq:ODE_synchronous_frame} and the theory of ODEs now tell us that the functions $E_\alpha(y)^\mu$ are $C^\infty$-bounded. Since these are the representations of the vectors $E_\alpha(y)$ in the basis $\{ e(y)(E_\alpha) \}$, we can conclude that the entries of the representations of the vectors $E_\alpha(y)$ in the basis $\{E_\alpha, E_\beta^\perp\}$ are $C^\infty$-bounded. But now these entries are exactly the first $(\dim E)$ columns of the change-of-basis matrix $\varphi(y)$.

Arguing analogously for the complement $E^\perp$, we get that the other columns of $\varphi(y)$ are also $C^\infty$-bounded, i.e., $\varphi$ itself is $C^\infty$-bounded.

It remains to show that the inverse homomorphism $\varphi^{-1}$ is $C^\infty$-bounded. But since pointwise it is given by the inverse matrix, i.e., $\varphi^{-1}(y) = \varphi(y)^{-1}$, this claim follows immediately from Cramer's rule, because we already know that $\varphi$ is $C^\infty$-bounded.
\end{proof}

\begin{proof}[Proof of point 2]
Let $\{E_\alpha(y)\}$ be a synchronous framing for $E$ and $\{E_\beta^\perp(y)\}$ one for $E^\perp$. Then $\{E_\alpha(y), E_\beta^\perp(y)\}$ is one for $E \oplus E^\perp$. Furthermore, let $\{v_i(y)\}$ be a synchronous framing for the trivial bundle $\IC^N$ and let $\varphi\colon E \oplus E^\perp \to \IC^N$ be the $C_b^\infty$-isomorphism.

Then projection matrix $e \in \Idem_{N \times N}(C^\infty(M))$ onto the subspace $E \oplus 0$ is given with respect to the basis $\{E_\alpha(y), E_\beta^\perp(y)\}$ of $E \oplus E^\perp$ and of $\IC^N$ by the usual projection matrix onto the first $(\dim E)$ vectors, i.e., its entries are clearly $C^\infty$-bounded since they are constant. Now changing the basis to $\{v_i(y)\}$, the representation of $e(y)$ with respect to this new basis is given by $\varphi^{-1}(y) \cdot e \cdot \varphi(y)$, i.e., $e \in \Idem_{N \times N}(C_b^\infty(M))$.
\end{proof}

If we have a $C_b^\infty$-complemented vector bundle $E$, then different choices of complements and different choices of isomorphisms with the trivial bundle lead to similar projection matrices. The proof of this is analogous to the corresponding proof in the usual case of vector bundles over compact Hausdorff spaces. We also get that $C_b^\infty$-isomorphic vector bundles produce similar projection matrices. Of course this also works the other way round, i.e., similar idempotent matrices give us $C_b^\infty$-isomorphic vector bundles. Again, the proof of this is the same as the one in the topological category.

\begin{defn}
Let $M$ be a manifold of bounded geometry. We define
\begin{itemize}
\item $\Vect_u(M)/_\sim$ as the abelian monoid of all complex vector bundles of bounded geometry over $M$ modulo $C_b^\infty$-isomorphism (the addition is given by the direct sum $[E] + [F] := [E \oplus F]$) and
\item $\Idem(C_b^\infty(M))/_\sim$ as the abelian monoid of idempotent matrizes of arbitrary size over the \Frechet algebra $C_b^\infty(M)$ modulo similarity (with addition defined as $[e] + [f] := \left[\begin{pmatrix}e&0\\0&f\end{pmatrix}\right]$).
\end{itemize}
These abelian monoids will be identified with each other in the following corollary.
\end{defn}

Let $f\colon M \to N$ be a $C^\infty$-bounded map\footnote{We use covers of $M$ and $N$ via normal coordinate charts of a fixed radius and demand that locally in this charts the derivatives of $f$ are all bounded and these bounds are independent of the chart used.} and $E$ a vector bundle of bounded geometry over $N$. Then it is clear that the pullback bundle $f^\ast E$ equipped with the pullback metric and connection is a vector bundle of bounded geometry over $M$.

The above discussion together with Proposition \ref{prop:image_proj_matrix_complemented} prove the following corollary:

\begin{cor}\label{cor:two_monoids_isomorphic}
The monoids $\Vect_u(M)/_\sim$ and $\Idem(C_b^\infty(M))/_\sim$ are isomorphic and this isomorphism is natural with respect to $C^\infty$-bounded maps between manifolds.\qed
\end{cor}

From this Corollary \ref{cor:two_monoids_isomorphic}, Lemma \ref{lem:C_b_infty_local} and Proposition \ref{prop:every_bundle_complemented} we immediately get the following interpretation of the $0$th uniform $K$-theory group $K^0_u(M)$ of a manifold of bounded geometry:

\begin{thm}[Interpretation of $K^0_u(M)$]\label{thm:interpretation_K0u}
Let $M$ be a Riemannian manifold of bounded geometry and without boundary.

Then every element of $K^0_u(M)$ is of the form $[E] - [F]$, where both $[E]$ and $[F]$ are $C_b^\infty$-isomorphism classes of complex vector bundles of bounded geometry over $M$.

Moreover, every complex vector bundle of bounded geometry over $M$ defines naturally a class in $K^0_u(M)$.\qed
\end{thm}

Note that the last statement in the above theorem is not trivial since it relies on the Proposition \ref{prop:every_bundle_complemented}.

\subsubsection*{Interpretation of \texorpdfstring{$K^1_u(M)$}{odd uniform K(M)}}

For the interpretation of $K^1_u(M)$ we will make use of suspensions of algebras. The suspension isomorphism theorem for operator $K$-theory states that we have an isomorphism $K_1(C_u(M)) \cong K_0(S C_u(M))$, where $S C_u(M)$ is the suspension of $C_u(M)$:
\begin{align*}
S C_u(M) & := \{ f \colon S^1 \to C_u(M) \ | \ f \text{ continuous and } f(1) = 0\}\\
& \cong \{ f \in C_u(S^1 \times M) \ | \ f(1, x) = 0 \text{ for all }x \in M\}.
\end{align*}
Equipped with the sup-norm this is again a $C^\ast$-algebra. Since functions $f \in S C_u(M)$ are uniformly continuous, the condition $f(1, x) = 0$ for all $x \in M$ is equivalent to $\lim_{t \to 1} f(t, x) = 0 \text{ uniformly in } x$.

Now in order to interpret $K_0(S C_u(M))$ via vector bundles of bounded geometry over $S^1 \times M$, we will need to find a suitable \Frechet subalgebra of $S C_u(M)$ so that we can again use Proposition \ref{prop:image_proj_matrix_complemented}. Luckily, this was already done by Phillips in \cite{phillips}:

\begin{defn}[Smooth suspension of a \Frechet algebras, {\cite[Definition 4.7]{phillips}}]\label{defn:smooth_suspension_algebra}
Let $A$ be a \Frechet algebra. Then the \emph{smooth suspension $S_\infty A$ of $A$} is defined as the \Frechet algebra
\[ S_\infty A := \{ f\colon S^1 \to A \ | \ f \text{ smooth and } f(1) = 0\}\]
equipped with the topology of uniform convergence of every derivative in every seminorm of $A$.
\end{defn}

For a manifold $M$ we have
\begin{align*}
S_\infty C_b^\infty(M) & \cong \{ f \in C_b^\infty(S^1 \times M) \ | \ f(1, x) = 0 \text{ for all }x \in M\} \\
& = \{ f \in C_b^\infty(S^1 \times M) \ | \ \forall k \in \IN_0 \colon \lim_{t \to 1} \nabla^k_x f(t, x) = 0 \text{ uniformly in } x\}.
\end{align*}

The proof of the following lemma is analogous to the proof of the Lemma \ref{lem:C_b_infty_local}:

\begin{lem}
Let $M$ be a manifold of bounded geometry.

Then the sup-norm completion of $S_\infty C_b^\infty(M)$ is $S C_u(M)$ and $S_\infty C_b^\infty(M)$ is a local $C^\ast$-algebra.\qed
\end{lem}

Putting it all together, we get $K^1_u(M) = K_0(S_\infty C_b^\infty(M))$, and Proposition \ref{prop:image_proj_matrix_complemented}, adapted to our case here, gives us the following interpretation of the $1$st uniform $K$-theory group $K^1_u(M)$ of a manifold of bounded geometry:

\begin{thm}[Interpretation of $K^1_u(M)$]\label{thm:interpretation_K1u}
Let $M$ be a Riemannian manifold of bounded geometry and without boundary.

Then every elements of $K^1_u(M)$ is of the form $[E] - [F]$, where both $[E]$ and $[F]$ are $C_b^\infty$-isomorphism classes of complex vector bundles of bounded geometry over $S^1 \times M$ with the following property: there is some neighbourhood $U \subset S^1$ of $1$ such that $[E|_{U \times M}]$ and $[F|_{U \times M}]$ are $C_b^\infty$-isomorphic to a trivial vector bundle with the flat connection (the dimension of the trivial bundle is the same for both $[E|_{U \times M}]$ and $[F|_{U \times M}]$).

Moreover, every pair of complex vector bundles $E$ and $F$ of bounded geometry and with the above properties define a class $[E] - [F]$ in $K_u^1(M)$.\qed
\end{thm}

Note that the last statement in the above theorem is not trivial since it relies on the Proposition \ref{prop:every_bundle_complemented}.

\subsection{Cap product}\label{sec:cap_product}

In this section we will define the cap product $\cap \colon K_u^p(X) \otimes K_q^u(X) \to K_{q-p}^u(X)$ for a locally compact and separable metric space $X$ of jointly bounded geometry\footnote{see Definition \ref{defn:jointly_bounded_geometry}}.

Recall that we have
\begin{equation*}
\LLip_R(X) := \{ f \in C_c(X) \ | \ f \text{ is }L\text{-Lipschitz}, \diam(\supp f) \le R \text{ and } \|f\|_\infty \le 1\}.
\end{equation*}

Let us first describe the cap product of $K_u^0(X)$ with $K^u_\ast(X)$ on the level of uniform Fredholm modules. The general definition of it will be given via dual algebras.

\begin{lem}\label{lem:proj_again_uniform_fredholm_module}
Let $P$ be a projection in $\Mat_{n \times n}(C_u(X))$ and let $(H, \rho, T)$ be a uniform Fredholm module.

We set $H_n := H \otimes \IC^n$, $\rho_n(-) := \rho(-) \otimes \id_{\IC^n}$, $T_n := T \otimes \id_{\IC^n}$ and by $\pi$ we denote the matrix $\pi_{ij} := \rho(P_{ij}) \in \Mat_{n \times n}(\IB(H)) = \IB(H_n)$.

Then $(\pi H_n, \pi \rho_n \pi, \pi T_n \pi)$ is a uniform Fredholm module, with an induced (multi-)grading if $(H, \rho, T)$ was (multi-)graded.
\end{lem}

\begin{proof}
Let us first show that the operator $\pi T_n \pi$ is a uniformly pseudolocal one. Let $R, L > 0$ be given and we have to show that $\{[\pi T_n \pi, \pi \rho_n(f) \pi] \ | \ f \in \LLip_R(X)\}$ is uniformly approximable. This means that we must show that for every $\varepsilon > 0$ there exists an $N > 0$ such that for every $[\pi T_n \pi, \pi \rho_n(f) \pi]$ with $f \in \LLip_R(X)$ there is a rank-$N$ operator $k$ with $\|[\pi T_n \pi, \pi \rho_n(f) \pi] - k\| < \varepsilon$.

We have
\[[\pi T_n \pi, \pi \rho_n(f) \pi] = \pi [T_n, \pi \rho_n(f)] \pi,\]
because $\pi^2 = \pi$ and $\pi$ commutes with $\rho_n(f)$. So since $(\pi \rho_n(f))_{ij} = \rho(P_{ij} f) \in \IB(H)$, we get for the matrix entries of the commutator
\[([T_n, \pi \rho_n(f)])_{ij} = [T, \rho(P_{ij} f)].\]

Since the $P_{ij}$ are bounded and uniformly continuous, they can be uniformly approximated by Lipschitz functions, i.e., there are $P_{ij}^\varepsilon$ with
\[\|P_{ij} - P_{ij}^\varepsilon\|_\infty < \varepsilon / (4n^2 \|T\|).\]
Note that we have $P_{ij}^\varepsilon f \in {L_{ij}\text{-}\operatorname{Lip}}_{R}(X)$, where $L_{ij}$ depends only on $L$ and $P_{ij}^\varepsilon$. We define $L^\prime := \max\{L_{ij}\}$.

Now we apply the uniform pseudolocality of $T$, i.e., we get a maximum rank $N^\prime$ corresponding to $R, L^\prime$ and $\varepsilon / 2n^2$. So let $k_{ij}^\varepsilon$ be the rank-$N^\prime$ operators corresponding to the functions $P_{ij}^\varepsilon f$, i.e.,
\[\|[T, \rho(P_{ij}^\varepsilon f)] - k_{ij}^\varepsilon\| < \varepsilon / 2n^2.\]

We set $k := \pi (k_{ij}^\varepsilon) \pi$, where $(k_{ij}^\varepsilon)$ is viewed as a matrix of operators. Then $k$ has rank at most $N := n^2 N^\prime$. Then we compute
\begin{align*}
\|[\pi T_n & \pi, \pi \rho_n(f) \pi] - k\|\\
& = \|\pi[T_n, \pi \rho_n(f)]\pi - \pi (k_{ij}^\varepsilon) \pi\|\\
& \le \|\pi\|^2 \cdot n^2 \cdot \max_{i,j}\{\|[T, \rho(P_{ij} f)] - k_{ij}^\varepsilon\|\}\\
& \le \|\pi\|^2 \cdot n^2 \cdot \max_{i,j}\{\underbrace{\|[T, \rho(P_{ij} f)] - [T, \rho(P_{ij}^\varepsilon f)]\|}_{=\|[T, \rho(P_{ij} - P_{ij}^\varepsilon)\rho(f)]\|} + \underbrace{\|[T, \rho(P_{ij}^\varepsilon f)] - k_{ij}^\varepsilon\|\}}_{\le \varepsilon / 2n^2}\\
& \le \|\pi\|^2 \cdot n^2 \cdot \max_{i,j}\{2 \|T\| \cdot \underbrace{\|\rho(P_{ij} - P_{ij}^\varepsilon)\| \cdot \|\rho(f)\|}_{\le \varepsilon/(4n^2 \|T\|)} + \varepsilon / 2n^2\}\\
& \le \|\pi\|^2 \cdot \varepsilon,
\end{align*}
which concludes the proof of the uniform pseudolocality of $\pi T_n \pi$.

That $(\pi T_n \pi)^2 - 1$ and $\pi T_n \pi - (\pi T_n \pi)^\ast$ are uniformly locally compact can be shown analogously. Note that because $T$ is uniformly pseudolocal we may interchange the order of the operators $T_n$ and $\rho(P_{ij}^\varepsilon f)$ in formulas (since for fixed $R$ and $L$ the subset $\{[T_n, \rho(P_{ij}^\varepsilon f)] \ | \ f \in \LLip_R(X) \} \subset \IB(H_n)$ is uniformly approximable).

We have shown that $(\pi H_n, \pi \rho_n \pi, \pi T_n \pi)$ is a uniform Fredholm module. That it inherits a (multi-)grading from $(H, \rho, T)$ is clear and this completes the proof.
\end{proof}

That the construction from the above lemma is compatible with the relations defining $K$-theory and uniform $K$-homology and that it is bilinear is quickly deduced and completely analogous to the non-uniform case. So we get a well-defined pairing
\[\cap\colon K_u^0(X) \otimes K_\ast^u(X) \to K_\ast^u(X)\]
which exhibits $K_\ast^u(X)$ as a module over the ring $K_u^0(X)$.\footnote{Compatibility with the internal product on $K_u^0(X)$, i.e., $(P \otimes Q) \cap T = P \cap (Q \cap T)$, is easily deduced. It mainly uses the fact that the isomorphism $\Mat_{n \times n}(\IC) \otimes \Mat_{m \times m}(\IC) \cong \Mat_{nm \times nm}(\IC)$ is canonical up to the ordering of basis elements. But different choices of orderings result in isomorphisms that differ by inner automorphisms, which makes no difference at the level of $K$-theory.}

To define the cap product in its general form, we will use the dual algebra picture of uniform $K$-homology, i.e., Paschke duality:

\begin{defn}[{\cite[Definition 4.1]{spakula_uniform_k_homology}}]\label{defn:frakD_frakC}
Let $H$ be a separable Hilbert space and $\rho \colon C_0(X) \to \IB(H)$ a representation.

We denote by $\frakD^u_{\rho \oplus 0}(X) \subset \IB(H \oplus H)$ the $C^\ast$-algebra of all uniformly pseudolocal operators with respect to the representation $\rho \oplus 0$ of $C_0(X)$ on the space $H \oplus H$ and by $\frakC^u_{\rho \oplus 0}(X) \subset \IB(H \oplus H)$ the $C^\ast$-algebra of all uniformly locally compact operators.
\end{defn}

That the algebras $\frakD^u_{\rho \oplus 0}(X)$ and $\frakC^u_{\rho \oplus 0}(X)$ are indeed $C^\ast$-algebras was shown by \Spakula in \cite[Lemma 4.2]{spakula_uniform_k_homology}. There it was also shown that $\frakC^u_{\rho \oplus 0}(X) \subset \frakD^u_{\rho \oplus 0}(X)$ is a closed, two-sided $^\ast$-ideal.

\begin{defn}
The groups $K_{-1}^u(X; {\rho \oplus 0})$ are analogously defined as $K_{-1}^u(X)$, except that we consider only uniform Fredholm modules whose Hilbert spaces and representations are (finite or countably infinite) direct sums of $H \oplus H$ and $\rho \oplus 0$.

For $K_0^u(X; {\rho \oplus 0})$ we consider only uniform Fredholm modules modeled on $H^\prime \oplus H^\prime$ with the representation $\rho^\prime \oplus \rho^\prime$, where $H^\prime$ is a finite or countably infinite direct sum of $H \oplus H$ and $\rho^\prime$ analogously a direct sum of finitely or infinitely many $\rho \oplus 0$, and the grading is given by interchanging the two summands in $H^\prime \oplus H^\prime$. Such Fredholm modules are called \emph{balanced} in \cite[Definition 8.3.10]{higson_roe}.
\end{defn}

\begin{prop}[{\cite[Proposition 4.3]{spakula_uniform_k_homology}}]\label{prop:paschke_duality}
The maps
\[\varphi_\ast \colon K_{1+\ast}(\frakD^u_{\rho \oplus 0}(X)) \to K_\ast^u(X; {\rho \oplus 0})\]
for $\ast = -1, 0$ are isomorphisms.
\end{prop}

Combining the above proposition with the following uniform version of Voiculescu's Theorem, we get the needed uniform version of Paschke duality.

\begin{thm}[{\cite[Corollary 3.6]{spakula_universal_rep}}]\label{thm:paschke_universal}
Let $X$ be a locally compact and separable metric space of jointly bounded geometry and $\rho\colon C_0(X) \to \IB(H)$ an ample representation, i.e., $\rho$ is non-degenerate and $\rho(f) \in \IK(H)$ implies $f \equiv 0$.

Then we have
\[K_\ast^u(X; \rho \oplus 0) \cong K_\ast^u(X)\]
for both $\ast = -1,0$.
\end{thm}

The following lemma is a uniform analog of the fact \cite[Lemma 5.4.1]{higson_roe} and is essentially proven in \cite[Lemma 5.3]{spakula_uniform_k_homology} (by ``setting $Z := \emptyset$'' in that lemma).

\begin{lem}\label{lem:K_theory_frakC_zero}
We have
\[K_\ast(\frakC^u_{\rho \oplus 0}(X)) = 0\]
and so the quotient map $\frakD^u_{\rho \oplus 0}(X) \to \frakD^u_{\rho \oplus 0}(X) / \frakC^u_{\rho \oplus 0}(X)$ induces an isomorphism
\begin{equation}\label{eq:K_theory_frakC_zero}
K_\ast(\frakD^u_{\rho \oplus 0}(X)) \cong K_\ast(\frakD^u_{\rho \oplus 0}(X) / \frakC^u_{\rho \oplus 0}(X))
\end{equation}
due to the $6$-term exact sequence for $K$-theory.\qed
\end{lem}

The last ingredient to construct the cap product is the inclusion
\begin{equation}\label{eq:commutator_C_u_with_frakD}
[C_u(X), \frakD^u_{\rho \oplus 0}(X)] \subset \frakC^u_{\rho \oplus 0}(X).
\end{equation}
It is proven in the following way: let $\varphi \in C_u(X)$ and $T \in \frakD^u_{\rho \oplus 0}(X)$. We have to show that $[\varphi, T] \in \frakC^u_{\rho \oplus 0}(X)$. By approximating $\varphi$ uniformly by Lipschitz functions we may without loss of generality assume that $\varphi$ itself is already Lipschitz. Now the claim follows immediately from $f[\varphi, T] = [f \varphi, T] - [f,T]\varphi$ since $T$ is uniformly pseudolocal.

Now we are able to define the cap product. Consider the map
\[ \sigma \colon C_u(X) \otimes \frakD^u_{\rho \oplus 0}(X) \to \frakD^u_{\rho \oplus 0}(X) / \frakC^u_{\rho \oplus 0}(X), \ f \otimes T \mapsto [fT].\]
It is a multiplicative $^\ast$-homomorphism due to the above Equation \eqref{eq:commutator_C_u_with_frakD} and hence induces a map on $K$-theory
\begin{equation*}
\sigma_\ast \colon K_\ast(C_u(X) \otimes \frakD^u_{\rho \oplus 0}(X)) \to K_\ast(\frakD^u_{\rho \oplus 0}(X) / \frakC^u_{\rho \oplus 0}(X)).
\end{equation*}
Using Paschke duality we may define the cap product as the composition
\begin{align*}
K_u^p(X) \otimes K_q^u(X; \rho \oplus 0) & \ = \ K_{-p}(C_u(X)) \otimes K_{1+q}(\frakD^u_{\rho \oplus 0}(X))\\
& \ \to \ \! K_{-p+1+q}(C_u(X) \otimes \frakD^u_{\rho \oplus 0}(X))\\
& \ \stackrel{\sigma_\ast}\to \ \! K_{-p+1+q}(\frakD^u_{\rho \oplus 0}(X) / \frakC^u_{\rho \oplus 0}(X))\\
& \stackrel{\eqref{eq:K_theory_frakC_zero}}\cong K_{-p+1+q}(\frakD^u_{\rho \oplus 0}(X))\\
& \ = \ K_{q-p}^u(X; \rho \oplus 0),
\end{align*}
where the first arrow is the external product on $K$-theory. So we get the cap product
\[\cap \colon K_u^p(X) \otimes K_q^u(X) \to K_{q-p}^u(X).\]

Let us state in a proposition some properties of it that we will need. The proofs of these properties are analogous to the non-uniform case.

\begin{prop}\label{prop:properties_general_cap_product}
The cap product has the following properties:
\begin{itemize}
\item the pairing of $K_u^0(X)$ with $K_\ast^u(X)$ coincides with the one in Lemma \ref{lem:proj_again_uniform_fredholm_module},
\item the fact that $K_\ast^u(X)$ is a module over $K_u^0(X)$ generalizes to
\begin{equation}\label{eq:general_cap_compatibility_module}
(P \otimes Q) \cap T = P \cap (Q \cap T)
\end{equation}
for all elements $P, Q \in K_u^\ast(X)$ and $T \in K_\ast^u(X)$, where $\otimes$ is the internal product on uniform $K$-theory,
\item if $X$ and $Y$ have jointly bounded geometry, then we have the following compatibility with the external products:
\begin{equation}\label{eq:compatibility_cap_external}
(P \times Q) \cap (S \times T) = (-1)^{qs} (P \cap S) \times (Q \cap T),
\end{equation}
where $P \in K_u^p(X)$, $Q \in K_u^q(Y)$ and $S \in K^u_s(X)$, $T \in K^u_t(Y)$, and
\item if we have a manifold of bounded geometry $M$, a vector bundle of bounded geometry $E \to M$ and an operator $D$ of Dirac type, then
\begin{equation}\label{eq:cap_twisted_Dirac}
[E] \cap [D] = [D_E] \in K_\ast^u(M),
\end{equation}
where $D_E$ is the twisted operator.\qed
\end{itemize}
\end{prop}

\subsection{Uniform \texorpdfstring{$K$-\Poincare}{K-Poincare }duality}\label{sec:poincare_duality}

We will prove in this section that uniform $K$-theory is \Poincare dual theory to uniform $K$-homology. This will be accomplished by a suitable Mayer--Vietoris induction.

\begin{thm}[Uniform $K$-\Poincare duality]
\label{thm:Poincare_duality_K}
Let $M$ be an $m$-dimensional spin$^c$ manifold of bounded geometry and without boundary.

Then the cap product $- \cap [M] \colon K_u^\ast(M) \to K^u_{m-\ast}(M)$ with its uniform $K$-fundamental class $[M] \in K_m^u(M)$ is an isomorphism.
\end{thm}

The proof of this theorem will occupy the whole subsection. We will first have to prove some auxiliary results before we will start on Page \pageref{page:proof_Poincare_duality} to assemble them into a proof of uniform $K$-\Poincare duality.

We will need the following Theorem \ref{thm:triangulation_bounded_geometry} about manifolds of bounded geometry. To state it, we have to recall some notions:

\begin{defn}[Bounded geometry simplicial complexes]\label{defn:simplicial_complex_bounded_geometry}
A simplicial complex has \emph{bounded geometry} if there is a uniform bound on the number of simplices in the link of each vertex.

A subdivision of a simplicial complex of bounded geometry with the properties that
\begin{itemize}
\item each simplex is subdivided a uniformly bounded number of times on its $n$-skeleton, where the $n$-skeleton is the union of the $n$-dimensional sub-simplices of the simplex, and that
\item the distortion $\operatorname{length}(e) + \operatorname{length}(e)^{-1}$ of each edge $e$ of the subdivided complex is uniformly bounded in the metric given by barycentric coordinates of the original complex,
\end{itemize}
is called a \emph{uniform subdivision}.
\end{defn}

\begin{defn}[Bi-Lipschitz equivalences]
Two metric spaces $X$ and $Y$ are said to be \emph{bi-Lipschitz equivalent} if there is a homeomorphism $f\colon X \to Y$ with
\[\tfrac{1}{C} d_X(x,x^\prime) \le d_Y(f(x), f(x^\prime)) \le C d_X(x,x^\prime)\]
for all $x,x^\prime \in X$ and some constant $C > 0$.
\end{defn}

\begin{thm}[{\cite[Theorem 1.14]{attie_classification}}]\label{thm:triangulation_bounded_geometry}
Let $M$ be a manifold of bounded geometry and without boundary.

Then $M$ admits a triangulation as a simplicial complex of bounded geometry whose metric given by barycentric coordinates is bi-Lipschitz equivalent to the metric on $M$ induced by the Riemannian structure. This triangulation is unique up to uniform subdivision.

Conversely, if $M$ is a simplicial complex of bounded geometry which is a triangulation of a smooth manifold, then this smooth manifold admits a metric of bounded geometry with respect to which it is bi-Lipschitz equivalent to $M$.
\end{thm}

\begin{rem}\label{rem:attie_regularity}
Attie uses in \cite{attie_classification} a weaker notion of bounded geometry as we do: additionally to a uniformly positive injectivity radius he only requires the sectional curvatures to be bounded in absolute value (i.e., the curvature tensor is bounded in norm), but he assumes nothing about the derivatives (see \cite[Definition 1.4]{attie_classification}). But going into his proof of \cite[Theorem 1.14]{attie_classification}, we see that the Riemannian metric constructed for the second statement of the theorem is actually of bounded geometry in our strong sense (i.e., also with bounds on the derivatives of the curvature tensor).

As a corollary we get that for any manifold of bounded geometry in Attie's weak sense there is another Riemannian metric of bounded geometry in our strong sense that is bi-Lipschitz equivalent the original one (in fact, this bi-Lipschitz equivalence is just the identity map of the manifold, as can be seen from the proof).
\end{rem}

\begin{lem}\label{lem:suitable_coloring_cover_M}
Let $M$ be a manifold of bounded geometry.

Then there is an $\varepsilon > 0$ and a countable collection of uniformly discretely distributed points $\{x_i\} \subset M$ such that $\{B_{\varepsilon}(x_i)\}$ is a uniformly locally finite cover of $M$.

Furthermore, it is possible to partition $\IN$ into a finite amount of subsets $I_1, \ldots, I_N$ such that for each $1 \le j \le N$ the subset $U_j := \bigcup_{i \in I_j} B_{\varepsilon}(x_i)$ is a disjoint union of balls that are a uniform distance apart from each other, and such that for each $1 \le K \le N$ the connected components of $U_K := U_1 \cup \ldots \cup U_k$ are also a uniform distance apart from each other (see Figure \ref{fig:not_allowed_cover}).
\end{lem}

\begin{figure}[htbp]
\centering
\includegraphics[scale=0.5]{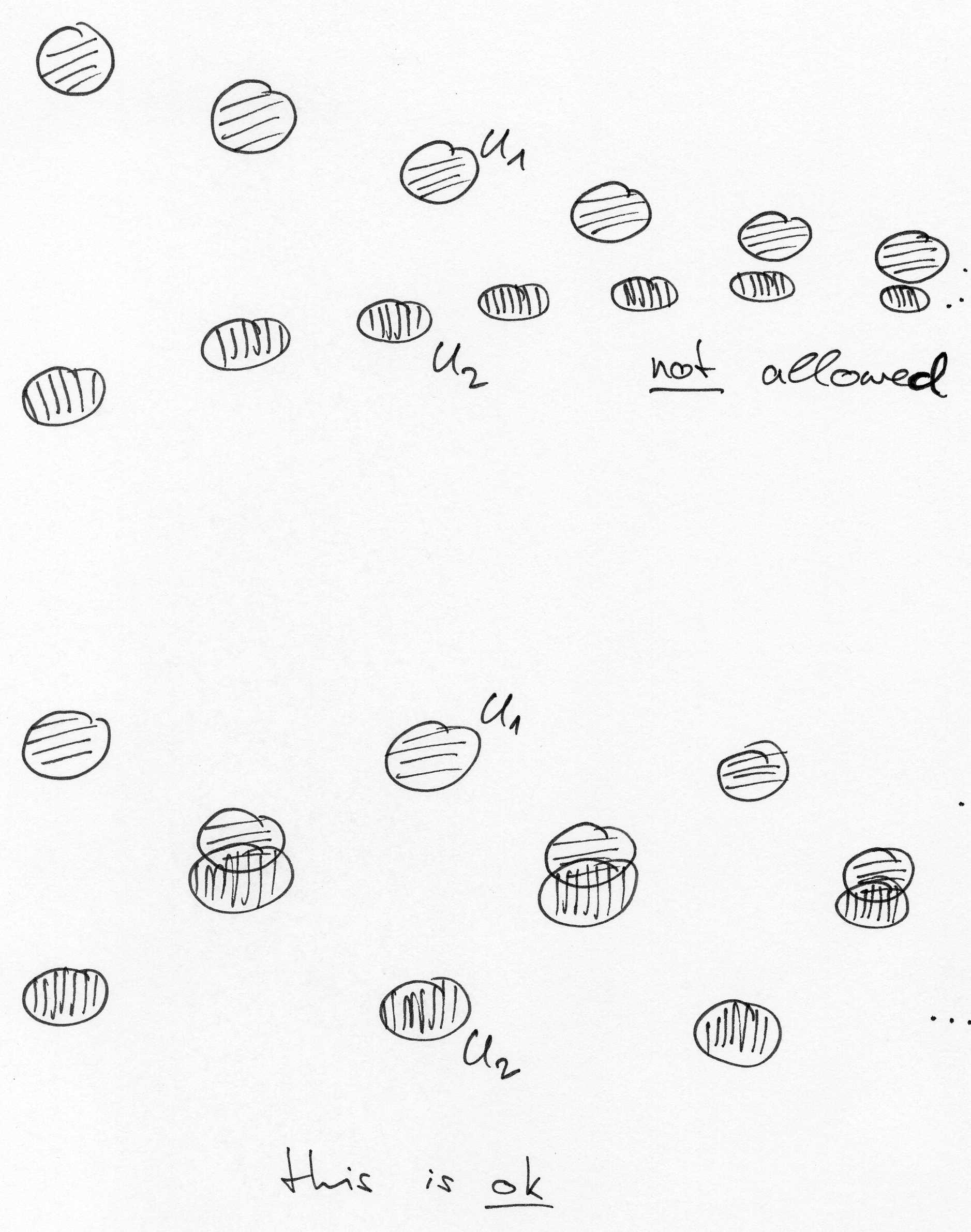}
\caption{Illustration for Lemma \ref{lem:suitable_coloring_cover_M}.}
\label{fig:not_allowed_cover}
\end{figure}

\begin{proof}
We triangulate $M$ via the above Theorem \ref{thm:triangulation_bounded_geometry}. Then we may take the vertices of this triangulation as our collection of points $\{x_i\}$ and set $\varepsilon$ to $2/3$ of the length of an edge multiplied with the constant $C$ which we get since the metric derived from barycentric coordinates is bi-Lipschitz equivalent to the metric derived from the Riemannian structure.

Two balls $B_\varepsilon(x_i)$ and $B_\varepsilon(x_j)$ for $x_i \not= x_j$ intersect if and only if $x_i$ and $x_j$ are adjacent vertices, and in the case that they are not adjacent, these balls are a uniform distance apart from each other. Hence it is possible to find a coloring of all these balls $\{B_\varepsilon(x_i)\}$ with finitely many colors having the claimed property: apply Lemma \ref{lem:coloring_graph} to the covering $\{B_\varepsilon(x_i)\}$ which has finite multiplicity due to bounded geometry.
\end{proof}

Our proof of \Poincare duality is a Mayer--Vietoris induction which will have only finitely many steps. So we first have to discuss the corresponding Mayer--Vietoris sequences.

We will start with the Mayer--Vietoris sequence for uniform $K$-theory. Let $O \subset M$ be an open subset, not necessarily connected. We denote by $(M, d)$ the metric space $M$ endowed with the metric induced from the Riemannian metric $g$ on $M$, and by $C_u(O, d)$ we denote the $C^\ast$-algebra of all bounded, uniformly continuous functions on $O$, where we regard $O$ as a metric space equipped with the subset metric induced from $d$ (i.e., we do not equip $O$ with the induced Riemannian metric and consider then the corresponding induced metric structure).

\begin{defn}\label{defn:uniform_k_th_subset}
Let $O \subset M$ be an open subset, not necessarily connected. We define $K^p_{u}(O \subset M) := K_{-p}(C_u(O,d))$.
\end{defn}

We will also need the following technical theorem:

\begin{lem}\label{lem:extension_samuel_compactification}
Let $O \subset M$ be open, not necessarily connected. Then every function $f \in C_u(O, d)$ has an extension to an $F \in C_u(M, d)$.
\end{lem}

\begin{proof}
For a metric space $X$ let $uX$ denote the Gelfand space of $C_u(X)$, i.e., this is a compactification of $X$ (the \emph{Samuel compactification}) with the following universal property: a bounded, continuous function $f$ on $X$ has an extension to a continuous function on $uX$ if and only if $f$ is uniformly continuous. We will use the following property of Samuel compactifications (see \cite[Theorem 2.9]{woods}): if $S \subset X \subset uX$, then the closure $\closure_{uX}(S)$ of $S$ in $uX$ is the Samuel compactification $uS$ of $S$.

So given $f \in C_u(O, d)$, we can extend it to a continuous function $\tilde{f} \in C(uO)$. Since $uO = \closure_{uM}(O)$, i.e., a closed subset of a compact Hausdorff space, we can extend $\tilde{f}$ by the Tietze extension theorem to a bounded, continuous function $\tilde{F}$ on $uM$. Its restriction $F := \tilde{F}|_M$ to $M$ is then a bounded, uniformly continuous function of $M$ extending $f$.
\end{proof}

\begin{lem}\label{lem:MV_uniform_K}
Let the subsets $U_j$, $U_K$ of $M$ for $1 \le j,K \le N$ be as in Lemma \ref{lem:suitable_coloring_cover_M}. Then we have Mayer--Vietoris sequences
\[\xymatrix{
K^0_{u}(U_K \cup U_{k+1} \subset M) \ar[r] & K^0_{u}(U_K \subset M) \oplus K^0_{u}(U_{k+1} \subset M) \ar[r] & K^0_{u}(U_K \cap U_{k+1} \subset M) \ar[d]\\
K^1_{u}(U_K \cap U_{k+1} \subset M) \ar[u] & K^1_{u}(U_K \subset M) \oplus K^1_{u}(U_{k+1} \subset M) \ar[l] & K^1_{u}(U_K \cup U_{k+1} \subset M) \ar[l]}\]
where the horizontal arrows are induced from the corresponding restriction maps.
\end{lem}

\begin{proof}
Recall the Mayer--Vietoris sequence for operator $K$-theory of $C^\ast$-algebras (see, e.g., \cite[Theorem 21.2.2]{blackadar}): given a commutative diagram of $C^\ast$-algebras
\[\xymatrix{P \ar[r]^{\sigma_1} \ar[d]_{\sigma_2} & A_1 \ar[d]^{\varphi_1}\\ A_2 \ar[r]^{\varphi_2} & B}\]
with $P = \{(a_1, a_2) \ | \ \varphi_1(a_1) = \varphi_2(a_2)\} \subset A_1 \oplus A_2$ and $\varphi_1$ and $\varphi_2$ surjective, then there is a long exact sequence (via Bott periodicity we get the $6$-term exact sequence)
\[\ldots \to K_n(P) \stackrel{({\sigma_1}_\ast,{\sigma_2}_\ast)}\longrightarrow K_n(A_1) \oplus K_n(A_2) \stackrel{{\varphi_2}_\ast - {\varphi_1}_\ast}\longrightarrow K_n(B) \to K_{n-1}(P) \to \ldots\]

We set $A_1 := C_u(U_K, d)$, $A_2 := C_u(U_{k+1}, d)$, $B := C_u(U_K \cap U_{k+1}, d)$ and $\varphi_1$, $\varphi_2$ the corresponding restriction maps. Due to the property of the sets $U_K$ as stated in the Lemma \ref{lem:suitable_coloring_cover_M} we get $P = C_u(U_K \cup U_{k+1}, d)$ and $\sigma_1$, $\sigma_2$ again just the restriction maps. To show that the maps $\varphi_1$ and $\varphi_2$ are surjective we  have to use the above Lemma \ref{lem:extension_samuel_compactification}.
\end{proof}

We will also need corresponding Mayer--Vietoris sequences for uniform $K$-homology. As for uniform $K$-theory we use here also the induced subspace metric (and not the metric derived from the induced Riemannian metric): let a not necessarily connected subset $O \subset M$ be given. We define $K_\ast^{u}(O \subset M)$ to be the uniform $K$-homology of $O$, where $O$ is equipped with the subspace metric from $M$, where we view $M$ as a metric space. The inclusion $O \hookrightarrow M$ is in general not a proper map (e.g., if $O$ is an open ball in a manifold) but this is no problem to us since we will have to use the wrong-way maps that exist for open subsets $O \subset M$: they are given by the inclusions $\LLip_R(O) \subset \LLip_R(M)$ for all $R, L > 0$. So we get a map $K_\ast^{u}(M) \to K_\ast^{u}(O \subset M)$ for every open subset $O \subset M$.

Existence of Mayer--Vietoris sequences for uniform $K$-homology of the subsets in the cover $\{U_K, U_{k+1}\}$ of $U_K \cup U_{k+1}$ (recall that we used Lemma \ref{lem:suitable_coloring_cover_M} to get these subsets) incorporating the wrong-way maps may be similarly shown as \cite[Section 8.5]{higson_roe}. The crucial excision isomorphism from that section may be constructed analogously as described in \cite[Footnote 73]{higson_roe}: for that construction Kasparov's Technical Theorem is used, and we have to use here in our uniform case the corresponding uniform construction which is as used in our construction of the external product for uniform $K$-homology.

Note that \Spakula constructed a Mayer--Vietoris sequence for uniform $K$-homology in \cite[Section 5]{spakula_uniform_k_homology}, but for closed subsets of a proper metric space. His arrows also go in the other direction as ours (since his arrows are induced by the usual functoriality of uniform $K$-homology).

We denote by $[M]|_O \in K_m^{u}(O \subset M)$ the class of the Dirac operator associated to the restriction to a neighbourhood of $O$ of the complex spinor bundle of bounded geometry defining the spin$^c$-structure of $M$ (i.e., we equip the neighbourhood with the induced spin$^c$-structure).

The cap product of $K^\ast_{u}(O \subset M)$ with $[M]|_O$ is analogously defined as the usual one, i.e., we get maps $- \cap [M]|_O \colon K^\ast_{u}(O \subset M) \to K_{m - \ast}^{u}(O \subset M)$. Now we have to argue why we get commutative squares between the Mayer--Vietoris sequences of uniform $K$-theory and uniform $K$-homology using the cap product. This is known for usual $K$-theory and $K$-homology; see, e.g., \cite[Exercise 11.8.11(c)]{higson_roe}. Since the cap product is in our uniform case completely analogously defined (see the second-to-last display before Proposition \ref{prop:properties_general_cap_product}), we may analogously conclude that we get commutative squares between our uniform Mayer--Vietoris sequences.

Let us summarize the above results:

\begin{lem}\label{lem:cap_prod_commutes_MV}
Let the subsets $U_j$, $U_K$ of $M$ for $1 \le j,K \le N$ be as in Lemma \ref{lem:suitable_coloring_cover_M}. Then we have corresponding Mayer--Vietoris sequences
\[\xymatrix{
K_0^{u}(U_K \cup U_{k+1} \subset M) \ar[r] & K_0^{u}(U_K \subset M) \oplus K_0^{u}(U_{k+1} \subset M) \ar[r] & K_0^{u}(U_K \cap U_{k+1} \subset M) \ar[d]\\
K_1^{u}(U_K \cap U_{k+1} \subset M) \ar[u] & K_1^{u}(U_K \subset M) \oplus K_1^{u}(U_{k+1} \subset M) \ar[l] & K_1^{u}(U_K \cup U_{k+1} \subset M) \ar[l]}\]
and the cap product gives the following commutative diagram:\vspace{\baselineskip}
\[\mathclap{\xymatrix{
K_{u}^\ast(U_{K} \cap U_{k+1} \subset M) \ar[r] \ar[d] & K_{u}^\ast(U_{K} \cup U_{k+1} \subset M) \ar[r] \ar[d] & K_{u}^\ast(U_K \subset M) \oplus K_{u}^\ast(U_{k+1} \subset M) \ar[d] \ar@/_1.5pc/[ll] \\
K^{u}_{m-\ast}(U_{K} \cap U_{k+1} \subset M) \ar[r] & K^{u}_{m-\ast}(U_{K} \cup U_{k+1} \subset M) \ar[r] & K^{u}_{m-\ast}(U_K \subset M) \oplus K^{u}_{m-\ast}(U_{k+1} \subset M) \ar@/^1.5pc/[ll]
}}\vspace{\baselineskip}\]
(We have suppressed the index shift due to the boundary maps in the latter diagram.)
\end{lem}

The last lemma that we will need before we will start to assemble everything together into a proof of uniform $K$-\Poincare duality is the following:

\begin{lem}\label{lem:uniform_k_hom_balls}
Let $M$ be an $m$-dimensional manifold of bounded geometry and let $U \subset M$ be a subset consisting of uniformly discretely distributed geodesic balls in $M$ having radius less than the injectivity radius of $M$ (i.e., each geodesic ball is diffeomorphic to the standard ball in Euclidean space $\IR^m$). Let the balls be indexed by a set $Y$.

Then we have $K_m^u(U \subset M) \cong \ell^\infty_\IZ(Y)$, the group of all bounded, integer-valued sequences indexed by $Y$, and $K_p^u(U \subset M) = 0$ for $p \not= m$.
\end{lem}

\begin{proof}
The proof is analogous to the proof of Lemma \ref{lem:uniform_k_hom_discrete_space}. It uses the fact that for an open ball $O \subset \IR^m$ we have $K_m(O \subset \IR^m) \cong \IZ$, and $K_p(O \subset \IR^m) = 0$ for $p \not= m$.
\end{proof}

\begin{proof}[Proof of uniform $K$-\Poincare duality]
\label{page:proof_Poincare_duality}
First we invoke Lemma \ref{lem:suitable_coloring_cover_M} to get subsets $U_j$ for $1 \le j \le N$.

The induction starts with the subsets $U_1$, $U_2$ and $U_1 \cap U_2$, which are collections of uniformly discretely distributed open balls, resp., in the case of $U_1 \cap U_2$ it is a collection of intersections of open balls, which is homotopy equivalent to a collection of uniformly discretely distributed open balls by a uniformly cobounded, proper and Lipschitz homotopy. Now uniform $K$-theory of a space coincides with the uniform $K$-theory of its completion, and furthermore, uniform $K$-theory is homotopy invariant with respect to Lipschitz homotopies. So the uniform $K$-theory of a collection of open balls is the same as the uniform $K$-theory of a collection of points. This groups we have already computed in Lemma \ref{lem:uniform_k_th_discrete_space}.

Uniform $K$-homology is homotopy invariant with respect to uniformly cobounded, proper and Lipschitz homotopies (see Theorem \ref{thm:homotopy_equivalence_k_hom}), and for totally bounded spaces it coincides with usual $K$-homology (see Proposition \ref{prop:compact_space_every_module_uniform}). So we have to compute uniform $K$-homology of a collection of uniformly discretely distributed open balls. This we have done in the above Lemma \ref{lem:uniform_k_hom_balls}.

Now we can argue that cap product is an isomorphism $K_u^\ast(U \subset M) \cong K_{m-\ast}^u(U \subset M)$, where $U$ is as in the above lemma. For this we have to note that if $M$ is a \spinc manifold, then the restriction of its complex spinor bundle to any ball of $U$ is isomorphic to the complex spinor bundle on the open ball $O \subset \IR^m$. This means that the cap product on $U$ coincides on each open ball of $U$ with the usual cap product on the open ball $O \subset \IR^m$. This all shows that we have \Poincare duality for the subsets $U_1$, $U_2$ and $U_1 \cap U_2$ (note that $U_1 \cap U_2$ is homotopic to a collection of open balls).

With the above Lemma \ref{lem:cap_prod_commutes_MV} we therefore get with the five lemma that the cap product is also an isomorphism for $U_1 \cup U_2$. The rest of the proof proceeds by induction over $k$ (there are only finitely many steps since we only go up to $k = N-1$), invoking every time the above Lemma \ref{lem:cap_prod_commutes_MV} and the five lemma. Note that in order to see that the cap product is an isomorphism on $U_K \cap U_{k+1}$, we have to write $U_K \cap U_{k+1} = (U_1 \cap U_{k+1}) \cup \ldots \cup (U_k \cap U_{k+1})$. This is a union of $k$ geodesically convex open sets and we have to do a separate induction on this one.
\end{proof}

\section{Final remarks and open questions}\label{sec_final_remarks}

In this paper we have defined and investigated uniform $K$-theory groups and uniform $K$-homology groups. But homotopy theory nowadays is practiced using spectra. So the question is whether one can refine our constructions here to the spectrum level such that the homotopy groups of the spectra coincide with the uniform $K$-theory and uniform $K$-homology groups.

One approach might be to consider something like uniform (co-)homology theories: one could try to put a model structure on the category of uniform spaces modeling uniform homotopy theory and then one could try to show that, e.g., uniform $K$-theory is nothing more but uniform homotopy classes of uniform maps into some uniform version of the $K$-theory spectrum.

Another approach might be to use $\infty$-categories and a motivic approach, similar as it was carried out in the case of coarse homology theories \cite{bunke_engel}.

\begin{question}\label{ques:uniform_theory}
Does there exist a reasonable uniform homotopy theory that recovers the uniform theories that we have considered in this article?
\end{question}

Baum and Douglas \cite{baum_douglas} defined a geometric version of $K$-homology, in which the cycles are \spinc manifolds with a vector bundle over them together a map into the space. This geometric picture is quite important for the understanding of index theory and so the question is whether we also have something similar for uniform $K$-homology.

\begin{question}\label{ques:geom_pic}
Is there a geometric picture of uniform $K$-homology?
\end{question}

A complete proof that geometric $K$-homology coincides on finite CW-complexes with analytic $K$-homology was given by Baum--Higson--Schick \cite{baum_higson_schick}. But this proof relies on a comparison of these theories with topological $K$-homology, i.e., with the homology theory defined by the $K$-theory spectrum. And this is now exactly the connection of Question \ref{ques:geom_pic} to Question \ref{ques:uniform_theory}.

\bibliography{./Bibliography_Uniform_K-th_and_PD}
\bibliographystyle{amsalpha}

\end{document}